\documentclass[final]{paper}
\usepackage{amsmath,amssymb,amsfonts}
\usepackage{dsfont}
\usepackage{stmaryrd}
\usepackage{url}
\usepackage{hyperref}
\usepackage{tikz} \usetikzlibrary{patterns}
\definecolor{bleuf}{rgb}{0.,0.,0.8}
\usepackage[color,notref,notcite]{showkeys}
\usepackage{enumitem}
\definecolor{labelkey}{rgb}{0.6,0,1}
\usepackage[normalem]{ulem}
\newcounter{corr}
\definecolor{violet}{rgb}{0.580,0.,0.827}
\newcommand{\corr}[3]{\typeout{Warning : a correction remains in page
\thepage}
				\stepcounter{corr}   \def\identite{{\mathrm{Id}}}     
				{\color{blue}\ifmmode\text{\,\sout{\ensuremath{#1}}\,}\else\sout{#1}\fi}
       {\color{red}#2}
       {\color{violet} #3}}

\newenvironment{proof}{\noindent {\sc Proof.} }{$\square$ }

\newtheorem{theorem}{Theorem}[section]
\newtheorem{definition}{Definition}[section]
\newtheorem{lemma}[theorem]{Lemma}
\newtheorem{remark}{Remark}[section]

\newcommand{\bfA}{{\boldsymbol A}}
\newcommand{\bfa}{{\boldsymbol a}}
\newcommand{\bfD}{{\boldsymbol D}}
\newcommand{\bfe}{{\boldsymbol e}}
\newcommand{\bfE}{{\boldsymbol E}}
\newcommand{\bff}{{\boldsymbol f}}
\newcommand{\bfg}{{\boldsymbol g}}
\newcommand{\bfn}{\boldsymbol n}
\newcommand{\bfT}{{\boldsymbol T}}
\newcommand{\bfu}{{\boldsymbol u}}
\newcommand{\bfv}{{\boldsymbol v}}
\newcommand{\bfV}{{\boldsymbol V}}
\newcommand{\bfVN}{\bfV_{\!\!N}}
\newcommand{\bfw}{{\boldsymbol w}}
\newcommand{\bfW}{{\boldsymbol W}}
\newcommand{\bfx}{\boldsymbol x}
\newcommand{\bfy}{\boldsymbol y}
\newcommand{\bfvarphi}{{\boldsymbol \varphi}}
\newcommand{\bfH}{{\boldsymbol H}}
\newcommand{\bfX}{\boldsymbol X}
\newcommand{\bfY}{\boldsymbol Y}
\newcommand{\bfC}{\boldsymbol C}
\newcommand{\bfZ}{\boldsymbol Z}

\newcommand{\dx}{\ \mathrm{d}\bfx}
\newcommand{\dt}{\ \mathrm{d} t}
\newcommand{\ds}{\ \mathrm{d} s}
\newcommand{\dedge}{\ \mathrm{d}\gamma(\bfx)}

\newcommand{\nstep}{N}
\newcommand{\nalgo}{n}
\newcommand{\nnn}{{n \in \xN}}
\newcommand{\nti}{n \to + \infty}
\newcommand{\Nti}{\nstep \to + \infty}
\newcommand{\deltat}{{\delta\!t}}
\newcommand{\disc}{{\mathcal D}}
\newcommand{\mesh}{{\mathcal M}}
\newcommand{\edge}{{\sigma}}
\newcommand{\edgeperp}{{\tau}}
\newcommand{\edges}{{\mathcal E}}
\newcommand{\edgesint}{{\mathcal E}_{\mathrm{int}}}
\newcommand{\edgesext}{{\mathcal E}_{\mathrm{ext}}}
\newcommand{\edgesinti}{{\mathcal E}_{\mathrm{int}}^{(i)}}
\newcommand{\edgesexti}{{\mathcal E}_{\mathrm{ext}}^{(i)}}
\newcommand{\edgesi}{{\edges\ei}}
\newcommand{\edgesj}{{\edges\ej}}
\newcommand{\edged}{\epsilon}
\newcommand{\edgesd}{{\widetilde {\edges}}}
\newcommand{\edgesdi}{{\edgesd^{(i)}}}
\newcommand{\edgesdinti}{{\edgesd^{(i)}_{{\rm int}}}}
\newcommand{\edgesdexti}{{\edgesd^{(i)}_{{\rm ext}}}}
\newcommand{\ei}{^{(i)}}
\newcommand{\ej}{^{(j)}}

\newcommand{\Hmesh}{\bfH_{\! N}}
\newcommand{\Hmeshzero}{\bfH_{\! N,0}}
\newcommand{\Hmeshmzero}{\bfH_{\edges_m,0}}
\newcommand{\Hmeshnzero}{\bfH_{\edges_n,0}}
\newcommand{\HmeshNzero}{\bfH_{\! N,0}}
\newcommand{\Hmeshi}{H_N^{(i)}}
\newcommand{\Hmeshij}{H_{\edgesd^{(i,j)}}}
\newcommand{\Hmeshizero}{H_{\edges^{(i)},0}}
\newcommand{\HmeshiN}{H_{N}^{(i)}}
\newcommand{\HmeshiNzero}{H_{N,0}^{(i)}}

\newcommand{\dive}{{\mathrm{div}}}
\newcommand{\gradi}{\boldsymbol \nabla}

\newcommand{\xN}{\mathbb{N}}
\newcommand{\xR}{\mathbb{R}}
 
\newcommand{\ie}{{\em i.e.}}
\newcommand{\blist}{\begin{list}{-}{\itemsep=0.5ex \topsep=0.5ex \leftmargin=1.cm \labelwidth=0.3cm \labelsep=0.5cm \itemindent=0.cm}}
\newcommand{\Ent}[1]{{\lfloor #1 \rfloor}}
\newcommand{\Char}{{\mathds 1}}
\newcommand{\exm}{^{(m)}}

\newcounter{cst}
\newcommand{\ctel}[1]{C_{\refstepcounter{cst}\label{#1}\thecst}}
\newcommand{\cter}[1]{C_{\ref{#1}}}

\author{ R. Eymard and D. Maltese\thanks{Universit\'e Gustave Eiffel, LAMA, (UMR 8050), UPEM, UPEC, CNRS, F-77454, Marne-la-Vallée (France), {\tt
robert.eymard, david.maltese@univ-eiffel.fr}}}
\title{Convergence of the incremental projection method \\ using conforming approximations}

\begin{document}

\maketitle


\begin{abstract}
We prove the convergence of an incremental projection numerical scheme for the time-dependent incompressible Navier--Stokes equations, without any regularity assumption on the weak solution. 
The velocity and the pressure are discretized in conforming spaces, whose compatibility is ensured by the existence of an interpolator for regular functions which preserves approximate divergence free properties. 
Owing to a priori estimates, we get the existence and uniqueness of the discrete approximation. Compactness properties are then proved, relying on a Lions-like lemma for time translate estimates. It is then possible to show the convergence of the approximate solution to a weak solution of the problem. The construction of the interpolator is detailed in the case of the lowest degree Taylor-Hood finite element.
\end{abstract}

%
%
\section{Introduction}
The Navier–Stokes equations for a homogeneous incompressible fluid can be written in a strong form as:
\begin{subequations} \label{pb:cont}
\begin{align}
 \label{qdm} &
 \partial_t \bfu + (\bfu \cdot \gradi)\bfu -  \Delta \bfu +\nabla p = \bff~\text{in}~(0, T) \times \Omega,
\\ \label{inc} &
\dive \bfu=0~\text{in}~(0, T) \times \Omega,
\end{align}
\end{subequations}
where the density and the viscosity are set to one for the sake of simplicity, and  where
\begin{equation}\label{hyp:T-Omega}
    \begin{array}{l}
        T>0, \mbox{ and }\Omega  \mbox{ is a connected, open and bounded subset of }\ \xR^d, \; d\in \{2, 3\},\\  \mbox{with a Lipschitz boundary } \partial \Omega.
    \end{array}
\end{equation}
The variables $\bfu$ and $p$ are respectively the velocity and the pressure of the fluid, and \eqref{qdm} and \eqref{inc} respectively model the momentum conservation and the mass conservation of an incompressible fluid. 
This system is supplemented with the homogeneous Dirichlet boundary condition 
\begin{equation}
\bfu=0~\text{on}~ (0,T) \times \partial \Omega,  \label{boundary}
\end{equation}%
and the initial condition 
\begin{equation}
\bfu(0)=\bfu_0~\text{in}~  \Omega.  \label{init}
\end{equation} 
The function $\bfu_0$ is the initial datum for the velocity and the function $\bff$ is the source term.

While this system of equations is coupling the velocity and the pressure, projection numerical schemes, introduced in  \cite{Chorin1969OnTC} and \cite{Temam1969SurLD}, enable the successive resolution of decoupled elliptic equations for the velocity and the pressure. This leads to cheaper and smaller computations than those issued from a coupled approximation.
Nevertheless, the projection method as originally presented has some drawbacks: the precision is limited for the $L^2$ norm of the pressure and for the $H^1$ norm of the velocity.
This limited precision  is a consequence of the artificial boundary condition for the pressure.
The use of a higher-order time discretization does not improve the convergence rate. We refer to  \cite{Rannacher1992OnCP} for error estimates for the semi-discrete case (see also \cite{Shen1992OnEE} and \cite{Shen1994RemarksOT}).
Different choices of the pressure boundary condition have been discussed to improve the efficiency of this method. The first major improvement of the original projection scheme was proposed in \cite{Goda1979AMT} (see also \cite{Kan1986ASA}) and now referred to (following \cite{Guermond2006AnOO}) as  incremental projection schemes. Such schemes can be written as follows in the case of a continuous-in-space discrete-in-time scheme:
\begin{itemize}
\item \emph{Prediction step}~: Find a function $\tilde \bfu^{n+1} \in H_0^1(\Omega)^d$, which is therefore regular in space and respects the boundary conditions but is not divergence-free, such that the following linearized momentum equation holds in the homogeneous Dirichlet weak sense:
\begin{multline} \label{pred2}
\frac{1}{\delta\!t}(\tilde \bfu^{n+1} - \bfu^n)+ \\ (\tilde \bfu^n \cdot \gradi)\tilde \bfu^{n+1} + \frac{1}{2} \dive \tilde \bfu^n \ \tilde \bfu^{n+1}
-  \Delta \tilde \bfu^{n+1} + \nabla p^n  = \bff^{n+1}~\text{in}~ \Omega.
\end{multline}
This step involves $d$ decoupled resolution of elliptic problems for each of the components of the velocity.
\item \emph{Correction step}~: Find the new pressure field $p^{n+1} \in H^1(\Omega)\cap L^2_0(\Omega)$ (denoting by $L^2_0(\Omega)$  the space of functions of $L^2(\Omega)$ with null average on $\Omega$), and a corrected velocity $ \bfu^{n+1} \in \bfV(\Omega)$ (denoting by $\bfV(\Omega)$ the space of $L^2$-divergence-free functions and null normal trace, precisely defined by \eqref{def:V}) therefore with a lower regularity in space and a weaker boundary condition, such that, in the homogeneous Neumann weak sense for the unknown $p^{n+1} -p^n\in H^1(\Omega)\cap L^2_0(\Omega)$,
\begin{subequations} \label{cor2}
\begin{align}
  &
 \frac{1}{\delta\!t}( \bfu^{n+1} - \tilde \bfu^{n+1}) +   \nabla (p^{n+1} -p^n) = 0~\text{in}~ \Omega,
\\  &
\dive  \bfu^{n+1}=0~\text{in}~ \Omega.
\end{align}
\end{subequations}

This step involves the resolution of an elliptic problem for finding the new pressure.
\end{itemize}
Scheme \eqref{pred2}-\eqref{cor2} is called an ``incremental projection scheme'', since it is obtained by introducing the previous pressure gradient in the prediction step and by solving the increment of the pressure in the correction step, which is a projection step of the predicted velocity on the divergence-free functions. Such a scheme seems to be much more efficient from a computational point of view than non-incremental projection schemes and has been the object of several error analyses, under some regularity assumptions on the solution of the continuous problem in the semi-discrete setting, see  \cite{Shen1992OnEE} and \cite{Shen1994RemarksOT}.
The incremental schemes have been the object of some works in the fully discrete setting (see  \cite{Guermond1998OnTA}, \cite{Guermond1998OnSA}, \cite{Guermond1996SomeIO}, \cite{Guermond1997CalculationOI}). For more details concerning projection methods we refer to \cite{Guermond2006AnOO} and references therein.

Some recent papers \cite{Kuroki2020OnCO,gal2022con} propose convergence proofs for fully discrete projection schemes. In \cite{gal2022con}, the incremental projection scheme is considered without any assumption of regularity assumptions on the exact solution. The proof of its convergence is done both for the semi-discrete scheme and for a fully discrete scheme using the Marker-And-Cell scheme (introduced in the seminal paper \cite{Harlow1965NumericalCO}). 
In this paper, we extend such a proof to the case where the spatial discretization is done using conforming methods (in most of the cases, finite element methods or spectral methods). 

For any integer $N\ge 1$, we define the time step by $\deltat_N = T/N$ and we consider approximation spaces $\bfX_{\!N} \subset H_0^1(\Omega)^d$ for the predicted velocity, $M_N\subset H^1(\Omega)\cap L^2_0(\Omega)$ for the approximate pressure. Then the space $\bfVN $, defined as the space of functions which are $L^2$-orthogonal to the gradient of the elements of $M_N$ (it is therefore not a subset of $\bfV(\Omega)$) is used for the approximation of the corrected velocity.  The fully discrete incremental projection scheme is given in Section \ref{sec:fulldisprojsch}, and is shown to have a unique solution $(\tilde \bfu_N^{\nalgo},p_N^{n},\bfu_{N}^{n})_{n=1,\ldots,N}\in (\bfX_{\!N}\times M_N\times \bfVN)^N$.

In order to prove the convergence of the incremental projection scheme, we need to prove the compactness of the sequence of the approximate velocities in $L^2$. The main difficulty relies in the time translate estimates; this difficulty is solved in the coupled case by the famous Lions' Lemma \cite[Lemme 5.1 p.58]{lions1969method}, bounding the norm in $L^2$ by a combination of the norm in $H^1$ and of a semi-norm similar to a $H^{-1}$ norm. In the case of the incremental projection method, the difficulty relies in the fact that there are two different approximations for the velocities. We handle this difficulty in Section \ref{sec:timetrans}, by closely following the method first introduced in \cite{gal2022con}. We can then conclude the convergence proof to a solution, considered in the weak sense given by Definition \ref{def:weaksol}, in Section \ref{sec:convweaksol}.

It is noticeable that, for this existence and uniqueness result, no compatibility condition between $M_N$ and $\bfX_{\!N}$ is required. This is no longer the case in Section \ref{sec:conv}, where the convergence of the method is proved. The compatibility condition is expressed through the existence of an interpolator $\Pi_N  : \bfW \to \bfX_N\cap \bfVN\cap L^\infty(\Omega)^d$, whose purpose is to build approximate values in the discrete velocity spaces  of some regular functions. This interpolator replaces in this paper a Fortin interpolator (see \cite{Scott1990FiniteEI,ern2017fin,Ern2021FiniteEI,gal2012fortin} for more properties on such operator). In the present paper, the requested properties are slightly different (see Section \ref{sec:cvhyp}):
\begin{itemize}
 \item The set $\bfW$  is a dense subset  of the set $H^1_0(\Omega)^d\cap \bfV(\Omega)$ for the $H^1_0(\Omega)^d$ norm (see \eqref{hyp:PNexist}).
 \item For any $\bfvarphi\in\bfW$, the convergence of $\Pi_N\bfvarphi$ to $\bfvarphi$ is assumed to hold in $H^1_0(\Omega)^d$, while the $L^\infty$ norm of $\Pi_N\bfvarphi$ remains bounded (see \eqref{hyp:PNconv}). 
\end{itemize}
So the hypotheses which are given on the space discretization in this paper are not shown to be equivalent to an inf-sup property. In Section \ref{sec:taylorhood}, we consider the example of the lowest degree Taylor-Hood finite element. In this section we provide all the calculations that prove the convergence property which is requested on $\Pi_N$, since the properties which are studied in the literature are generally different. We show that constructing $\Pi_N$ and checking its properties is much simplified by the fact that we only need to apply it on regular divergence free functions with compact support. It is in some way simpler than the verification of the inf-sup condition for these spaces (we provide in this section simple examples where there is no inf-sup stability with the lowest degree Taylor-Hood finite element, but where $\Pi_N$ is nevertheless defined). Note that the study done in this paper also applies to the semi-discrete scheme (see Remark \ref{rem:semidiscrete}). 
We refer to \cite{gal2022con} for the convergence of the semi-discrete scheme using another discretization of the convective term.

Throughout the paper, we shall assume that the data $\bff$ and $\bfu_0$ satisfy
\begin{equation}\label{hyp:f-u0}
\bff \in L^2(0,T;H^{-1}(\Omega)^d)  \mbox{ and }\bfu_0 \in L^2(\Omega)^d.
\end{equation}
We denote in the whole paper
\begin{equation}\label{eq:pdtscaldd}
 \forall \bfv,\bfw\in L^2(\Omega)^d,\ (\bfv,\bfw) := \int_\Omega \bfv(\bfx)\cdot\bfw(\bfx)\dx
\end{equation}
and
\begin{equation}\label{eq:pdtduality}
  \forall \bfg\in H^{-1}(\Omega)^d,\ \forall \bfv\in H^1_0(\Omega)^d,\ \langle\bfg,\bfv\rangle := \bfg(\bfv).
\end{equation}

We recall that we denote  by $L^2_0(\Omega)$ the space of $L^2$ functions with null average on $\Omega$, which can be defined by
  \begin{equation}\label{def:Ldz}
L^2_0(\Omega)= \{ q \in L^2(\Omega)~\text{such that}~\int_\Omega q(\bfx) \dx=0\},
\end{equation}
and by $\bfV(\Omega)$ the space of $L^2$-divergence-free functions, which can be defined by
  \begin{equation}\label{def:V}
\bfV(\Omega)= \{ \bfu \in L^2(\Omega)^d~\text{such that}~( \bfu , \nabla\xi ) =0~\text{for any}~\xi \in H^1(\Omega)\}.
\end{equation}
Recall that $\bfV(\Omega)$ (sometimes also denoted by $L^2_\sigma(\Omega)$) is the closure in $L^2(\Omega)$ of $C^\infty_{c,\sigma}(\Omega) := \{\bfv\in C^\infty_c(\Omega)^d;\dive\bfv = 0\}$ 
and that we can write as well \cite[Eq. 3.5.13]{sohr2001nav}, \cite[Definition IV.3.2 p. 248]{BoyerFabrie-book}
\[
 \bfV(\Omega)= \{ \bfu \in H_{\rm div}(\Omega)~\text{such that}~\dive \bfu = 0 \text{ and }\gamma_\nu \bfu = 0\},
\]
denoting by $\gamma_\nu$ the normal trace of $\bfu$ on $\partial\Omega$.
Then $H_0^1(\Omega)^d\cap\bfV(\Omega)$  (sometimes also denoted by $W^{1,2}_{0,\sigma}(\Omega)$) is the subset of $H_0^1(\Omega)^d$ of divergence-free functions, and is the closure in $H^1(\Omega)$ of $C^\infty_{c,\sigma}(\Omega)$  \cite[p.157]{sohr2001nav}, \cite[Lemma IV.3.4 p.249]{BoyerFabrie-book}.
Let us define the weak solutions of Problem  \eqref{pb:cont}-\eqref{init} in the sense of Leray \cite{leray1934}.
\begin{definition}[Weak solution]\label{def:weaksol}
Under the assumptions \eqref{hyp:T-Omega} and \eqref{hyp:f-u0}, a function $\bfu \in L^2(0,T;H_0^1(\Omega)^d\cap\bfV(\Omega))$ $\cap$ $L^\infty(0,T;L^2(\Omega)^d)$ is a weak solution of the  problem \eqref{pb:cont}-\eqref{init} if 
\begin{multline} \label{weaksol}
- \int_0^T \!\!(\bfu, \partial_t \bfv) \dt + \int_0^T \!\!((\bfu \cdot \gradi) \bfu,  \bfv) \dt   + \int_0^T \!\!\int_\Omega \gradi \bfu  : \gradi \bfv \dx \dt \\
= (\bfu_0 , \bfv(0,\cdot) ) +  \int_0^T \!\!\langle \bff , \bfv\rangle \dt
\end{multline}
for any  $\bfv\in C_c^\infty (\Omega \times [0,T))^d$, such that $\dive \bfv=0$ a.e. in $\Omega \times (0,T)$.
\end{definition}

\section{The projection scheme using a conforming method}\label{sec:fulldisprojsch}

\subsection{Space and time discretizations}
We consider a partition of the time interval $[0,T]$, which we suppose uniform to alleviate the notations, so that the assumptions read (omitting in the whole paper to recall that $N$ is assumed to be an integer):
\begin{equation}\label{hyp:timediscfinite}
  N\ge 1, \qquad \deltat_{\! N} = \frac T N, \qquad t_N^n  = n\,\deltat_{\! N} \mbox{ for }  n \in \llbracket 0,N\rrbracket.\end{equation}
 We consider a sequence of velocity-pressure approximations in $H_0^1(\Omega)^d$ and $H^1(\Omega) \cap L^2_0(\Omega)$ respectively.
For the approximation of the predicted velocity, let
\begin{equation}\label{hyp:X}
  \begin{array}{l}
\mbox{ $(\bfX_{\!N}  )_{N \ge 1}$ be a sequence of closed subspaces of $H^1_0(\Omega)^d$.}
  \end{array}
\end{equation}
For the approximation of the pressure, let  
\begin{equation}\label{hyp:M}
  \begin{array}{l}
\mbox{ $(M_N )_{N \ge 1}$ be a  sequence of closed subspaces of $H^1(\Omega) \cap L^2_0(\Omega)$.}
\end{array}
\end{equation}
For the approximation of the corrected velocity
we define the space of weakly divergence free functions by
\begin{equation}\label{eq:defVN}
\bfVN    = \{ \bfv \in L^2(\Omega)^d~\text{such that}~( \bfv ,\nabla q )= 0~\text{for any}~q \in M_N \}.
\end{equation}

We denote by $\mathcal{P}_{\bfVN   } : L^2(\Omega)^d \to \bfVN   $ the orthogonal projection in $L^2(\Omega)^d$ onto the space $\bfVN   $. We denote by $\mathcal{P}_{\bfV(\Omega)} : L^2(\Omega)^d \to \bfV(\Omega)$ the orthogonal projection in $L^2(\Omega)^d$ onto the space $\bfV(\Omega)$.

\begin{remark}\label{rem:semidiscretezero} The choice $\bfX_{\!N} = H^1_0(\Omega)^d$, $M_N = H^1(\Omega) \cap L^2_0(\Omega)$ (which implies $\bfVN  = \bfV(\Omega)$) yields the continuous-in-space discrete-in-time scheme. 
\end{remark}

\subsection{The projection scheme}
 Let 
$a : H_0^1(\Omega)^d \times H_0^1(\Omega)^d \to \mathbb{R}$ be the coercive bilinear form defined by
\[
 a(\bfu,\bfv) = \int_\Omega \gradi \bfu : \gradi \bfv \dx,
\]
and let $b : H_0^1(\Omega)^d  \times H_0^1(\Omega)^d \times H_0^1(\Omega)^d \to \mathbb{R} $  be the ``skew-symmetric'' trilinear form defined by
\[
b(\bfu,\bfv,\bfw) =  \frac{1}{2} \int_\Omega \Big((\bfu \cdot \gradi) \bfv \cdot \bfw - (\bfu \cdot \gradi) \bfw \cdot \bfv\Big) \dx. 
\]
Note that we have for any $(\bfu,\bfv,\bfw) \in  H_0^1(\Omega)^d  \times H_0^1(\Omega)^d \times H_0^1(\Omega)^d $ 
\begin{equation}\label{eq:identityconv}
b(\bfu,\bfv,\bfw)  =\int_\Omega (\bfu \cdot \gradi) \bfv \cdot \bfw \dx + \frac{1}{2} \int_\Omega \dive \bfu ( \bfv \cdot \bfw) \dx.
\end{equation}
Note that the fact that $b$ is continuous over  $H_0^1(\Omega)^d  \times H_0^1(\Omega)^d \times H_0^1(\Omega)^d$ is a consequence of the Sobolev continuous embedding of $H_0^1(\Omega)$ in $L^4(\Omega)$ for $d \in \{2,3\}$.
Using the above definitions, the first order time fully discrete incremental projection scheme reads:  
\begin{subequations}\label{eq:semidiscretefinite}  
\begin{align} 
 & \mbox{\emph{Initialization:}} \nonumber \\
 &\hspace{2ex} \mbox{Let}~ \bfu_N^0 = \mathcal{P}_{\bfVN} \bfu_0,~\tilde \bfu_N^n=0 ~\text{and}~ p_N^0 =0. \label{eq:semidisc-finiteinit}\\[2ex]
 &\mbox{Solve for } 0 \le n \le N-1: \nonumber \\
 & \hspace{2ex}\mbox{\emph{Prediction step:}} ~\mbox{Find $\bfu_N^{\nalgo+1} \in \bfX_{\!N}$ such that }  \nonumber \\
  & \label{eq:finitepre}  \hspace{2ex}
\frac{1}{\deltat_{\! N}} \Big( (\tilde{\bfu}_N^{n+1},\bfv) - (\bfu_N^\nalgo,\bfv)\Big) +b( \tilde \bfu_N^n,\tilde \bfu_N^{n+1},\bfv)
\\[1ex] \nonumber 
& \hspace{1ex} +a(\tilde{\bfu}_N^{n+1},\bfv) +(\nabla p_N^n,\bfv) = \langle\bff_{\! N}^{n+1},\bfv\rangle,~\text{for any}~\bfv \in \bfX_{\!N}  ,\\[1ex] \nonumber 
  &\hspace{2ex} \mbox{\emph{Correction step 1:}}~\mbox{Find $p_N^{\nalgo+1}  \in  M_{\!N}$ such that }  \nonumber \\ 
  \label{eq:finitecor}  
&\hspace{2ex} 
(\nabla (p_N^{n+1}-p_N^n),\nabla q) = \frac{1}{\delta\!t_N}( \tilde \bfu_N^{n+1},\nabla q),~\text{for any}~q \in M_N(\Omega).
  \\[1ex] \nonumber 
  &\hspace{2ex} \mbox{\emph{Correction step 2:}}~\mbox{Define $\bfu_N^{\nalgo+1} \in L^2(\Omega)^d$ as follows }  \nonumber \\ 
  \label{eq:finitecor2}  
&\hspace{2ex} 
\bfu_N^{n+1} = \tilde \bfu_N^{n+1} - {\delta\!t_N} \nabla (p_N^{n+1} -p_N^n).
\end{align}
\end{subequations}
The term $\bff_{\! N}^{n+1} \in H^{-1}(\Omega)^d$ is given by 
$  \bff_{\! N}^{n+1}=\dfrac 1 {\deltat_{\! N}}\int_{t^n}^{t^{n+1}} \bff(t)  \dt.$
\begin{remark}\label{rem:rembothfields}
Note that $\bfu_N^{n+1}$ belongs to $\bfX_N + \nabla M_N\subset L^2(\Omega)^d$.
We refer to \cite{Guermond1996SomeIO} for other discretizations of the projection step. Multiplying \eqref{eq:finitecor2} by the gradient of an element of $M_N$ and using \eqref{eq:finitecor2} we obtain for any $n \in \llbracket 0,N-1 \rrbracket$
\begin{equation}\label{eq:zerodiv}
(\bfu_N^{n+1},\nabla q)=0~\text{for any}~q\in M_N,
\end{equation}
which means that $\bfu_N^{n+1}\in\bfVN $. Hence, as stated in the introduction, $\bfu_N$ satisfies the discrete divergence free constraint, but not the regularity one. We therefore obtain a discrete Helmholtz-Leray decomposition of the function  $\tilde \bfu_N^{n+1}$ by writing $\tilde \bfu_N^{n+1} = \bfu_N^{n+1} + \delta\!t_N \nabla(p_N^{n+1} - p_N^n)$. Although $\tilde \bfu_N$ safisfies the regularity constraint and the boundary conditions, but not the discrete divergence free constraint, Theorem \ref{thm:cvscheme} proves that both approximate fields converge to the same weak solution of the continuous problem, simultaneously satisfying all these constraints.
\end{remark}
\begin{remark}
 Note that the numerical method can be expressed without using the sequence  of corrected velocity $(\bfu_N^n)_{n \in \llbracket 0,N\rrbracket }$. Letting $p_N^0 = 0$, the predicted velocity is such that
 $$
 \frac{1}{\deltat_{\! N}} \Big( (\tilde{\bfu}_N^{1},\bfv) - (\mathcal{P}_{\bfVN} \bfu_0,\bfv)\Big) + a(\tilde \bfu_N^1,\bfv)= \langle\bff_{\! N}^{1},\bfv\rangle,~\text{for any}~\bfv \in \bfX_{\!N}, 
 $$
and, for $n \in \llbracket 1,N-1 \rrbracket$,
 \begin{multline*}
 \frac{1}{\deltat_{\! N}} \Big( (\tilde{\bfu}_N^{n+1},\bfv) - (\tilde \bfu_N^n,\bfv)\Big) + b( \tilde \bfu_N^n,\tilde \bfu_N^{n+1},\bfv) \\
 +a(\tilde \bfu_N^{n+1},\bfv) + (\nabla (2 p_N^n - p_N^{n-1}),\bfv) = \langle\bff_{\! N}^{n+1},\bfv\rangle,~\text{for any}~\bfv \in \bfX_{\!N}.
 \end{multline*}
The pressure $p_N^{n+1} \in M_N$ is such that 
\begin{equation*}
(\nabla (p_N^{n+1}-p_N^n),\nabla q) = \frac{1}{{\delta\!t_N}}( \tilde \bfu_N^{n+1},\nabla q)~\text{for any}~q\in M_N.
\end{equation*}
\end{remark}
\begin{remark}
The choice for the initial conditions $\tilde \bfu_N^0$ and $p_N^0$ plays no significant role for the convergence study. It suffices to consider any sequences $(\tilde \bfu_N^0)_{N \ge 1}$ bounded in $H^1(\Omega)^d$ and  $(\tilde p_N^0)_{N \ge 1}$ bounded in $H^1(\Omega)$.
\end{remark}

\subsection{Existence of a solution to the projection scheme}

The following existence result allows to define the approximate solutions obtained by the projection scheme \eqref{eq:semidiscretefinite}.
\begin{lemma}[Approximate solutions] \label{lem:exist-finite}
 Under Assumptions \eqref{hyp:T-Omega}, \eqref{hyp:f-u0}, \eqref{hyp:timediscfinite}, \eqref{hyp:X}, \eqref{hyp:M}, for any $N\ge 1$,
 there exists a unique sequence $(\tilde \bfu_N^{\nalgo},\bfu_N^{\nalgo},$ $p_N^{\nalgo})_{\nalgo \in \llbracket 1,N \rrbracket} \subset$  $\bfX_{\!N}  \times \bfVN  $ $ \times M_N $ satisfying \eqref{eq:semidiscretefinite}. 
 We then define the functions  $\bfu_\nstep : (0,T) \to \bfVN   $ and $\tilde \bfu_\nstep: (0,T) \to \bfX_{\!N}  $ by
\begin{equation}\label{eqdef:fullfunctions-finite}
\bfu_\nstep (t) = \sum_{n=0}^{\nstep-1} \ \mathds 1_{(t^\nalgo_\nstep,t^{\nalgo+1}_\nstep]}(t)\bfu^{\nalgo}_\nstep,\quad \quad
\tilde \bfu_\nstep (t) = \sum_{n=0}^{N-1} \  \mathds 1_{(t^\nalgo_\nstep,t^{n+1}_\nstep]}(t) \tilde \bfu^{\nalgo+1}_\nstep.
\end{equation}
 \end{lemma}

 \begin{proof}
  For given functions $p_N^\nalgo\in M_N $, $\tilde \bfu_N^{\nalgo}\in \bfX_{\!N}  $ and $\bfu_N^\nalgo\in \bfVN   $, Problem \eqref{eq:finitepre} is under the form:
  \[
   \tilde \bfu_N^{\nalgo+1}\in \bfX_{\!N}  ,\ \forall \bfv\in \bfX_{\!N}  ,\ a_N^\nalgo(\bfu_N^{\nalgo+1},\bfv) = \langle \bfg_N^n, \bfv\rangle,
  \]
  where we define the bilinear form $a_N^\nalgo:H_0^1(\Omega)^d \times H_0^1(\Omega)^d\to\mathbb{R}$,
  $(\bfu,\bfv)\mapsto \frac{1}{\deltat_{\! N}} (\bfu,\bfv)  +b( \tilde \bfu_N^n,\bfu,\bfv) + a(\bfu,\bfv)$, and where $\bfg_N^n\in H^{-1}(\Omega)^d$. The continuity of this bilinear form is a consequence of that of $b$.   
  Owing to the property $b(\cdot,\bfv,\bfv) = 0$, this continuous bilinear form is such that
  \[
   \forall \bfv\in \bfX_{\!N}  ,\ a_N^\nalgo(\bfv,\bfv) \ge \Vert \bfv\Vert_{H^1_0(\Omega)^d}^2.
  \]
Hence the existence and uniqueness of the solution $\tilde \bfu_N^{\nalgo+1}\in \bfX_{\!N}  $ (which is closed by assumption) is given by the Lax-Milgram theorem.
  
 For given functions  $p_N^\nalgo\in M_N$, $\tilde{\bfu}_N^{n+1}\in \bfX_{\!N}$, Problem \eqref{eq:finitecor} is under the form:
  \[
  p_N^{\nalgo+1}\in M_N ,\ \forall q \in M_N ,\ (\nabla p_N^{n+1}, \nabla q) = c_N(q).
  \]
  Hence the existence and uniqueness of the solution $p_N^{\nalgo+1}\in M_N $ (again closed by assumption) is given by the fact that the bilinear form $(p,q)\mapsto (\nabla p, \nabla q)$ is coercive on $H^1(\Omega) \cap L_0^2(\Omega)$. 
 \end{proof}

\section{Convergence study}\label{sec:conv}

When the Stokes problem is formulated as a coupled mixed problem whose unknowns are the velocity and the pressure, the approximation spaces used in the case of a coupled scheme must satisfy a compatibility condition (called the inf–sup condition) to obtain existence and uniqueness of a solution (see \cite{brezzifortin}). Although such condition is not needed for proving Lemma \ref{lem:exist-finite}, the numerical simulations presented in \cite{Guermond1998OnSA} suggest that some compatibility conditions are necessary to obtain satisfactory convergence properties. We present in Section \ref{sec:cvhyp} some sufficient compatibility conditions, which are not equivalent to the standard inf-sup condition, as detailed in Section \ref{sec:taylorhood}.

\subsection{Convergence assumptions}\label{sec:cvhyp}

The following assumptions are done on the sequences  $(\bfX_{\!N}  )_{N\ge 1}$ and $(M_N )_{N\ge 1}$ in order to prove the convergence of the projection scheme as $N\to\infty$. These assumptions, provided in terms of the existence of a family  $(\Pi_N)_{N\ge 1}$ of interpolators on some spaces, replace the standard uniform inf-sup hypothesis.
We assume that
\begin{equation}\label{hyp:PNexist}
  \begin{array}{l}
\mbox{ $\bfW$ is a dense subset of $ \bfV(\Omega)\cap H^1_0(\Omega)^d$ for the norm of $H^1_0(\Omega)^d$. } 
\end{array}
\end{equation}
We define
\begin{equation}\label{eq:defbfE}
 \forall N\ge 1,\ \bfE_N = \bfX_{\!N} \cap \bfVN\cap L^\infty(\Omega)^d,
\end{equation}
\begin{equation}\label{eq:defnormeE}
 \forall N\ge 1,\ \forall \bfu\in \bfE_N, \ \Vert \bfu\Vert_\bfE = \Vert\gradi\bfu\Vert_{L^2(\Omega)^{d\times d}} + \Vert\bfu\Vert_{L^\infty(\Omega)^{d}}.
\end{equation}
We assume that
\begin{equation}\label{hyp:PNconv}
  \begin{array}{l}
\mbox{ For any $N \ge 1$ there exists a mapping $\Pi_N : \bfW \to \bfE_N  $ such that,} \\
 \mbox{for any $\bfvarphi\in  \bfW $,   $(\Pi_N\bfvarphi )_{N \ge 1}$ converges to  $\bfvarphi$ in $H^1_0(\Omega)^d$  } \\
 \mbox{and the sequence  $(\Vert \Pi_N\bfvarphi\Vert_\bfE )_{N \ge 1}$ is bounded.} \\
  \end{array}
\end{equation}

 The sequence $(M_N )_{N\ge 1}$ is assumed to be such that
 \begin{equation}\label{hyp:Mlim}
  \begin{array}{l}
 \mbox{  $(\mathcal{P}_{M_N }(p))_{N \ge 1}$ converges to $p$ in $H^1(\Omega)$, } \\
 \mbox{for any $p \in H^1(\Omega) \cap L^2_0(\Omega)$,}
  \end{array}
\end{equation}
where we denote by $\mathcal{P}_{M_{N} } : H^1(\Omega) \cap L^2_0(\Omega) \to M_{N} $  the orthogonal projection  onto the space $M_{N} $ in the Hilbert space $H^1(\Omega) \cap L_0^2(\Omega)$  equipped with the scalar product $(p,q)\mapsto (\nabla p, \nabla q)$. 

\begin{remark}
 The density of $C^\infty_c(\Omega)^d\cap\bfV(\Omega)$ in  $\bfV(\Omega)$ for the $L^2(\Omega)^d$ norm is proved in particular in \cite[Lemma IV.3.5 p. 249]{BoyerFabrie-book}. Since the density of $\bfW$ in  $H^1_0(\Omega)^d\cap \bfV(\Omega)$ enables to approximate any element of $C^\infty_c(\Omega)^d\cap\bfV(\Omega)$ as closely as desired  for the $L^2(\Omega)^d$ norm, we get that $\bfW$  is dense as well in  $\bfV(\Omega)$ for the $L^2(\Omega)^d$ norm. 
\end{remark}

\begin{remark}\label{rem:semidiscrete} If, for all $N\ge 1$, we consider as in Remark \ref{rem:semidiscretezero} the semi-discrete case $\bfX_N = H^1_0(\Omega)^d$, $M_N = H^1(\Omega) \cap L^2_0(\Omega)$ (recall that this yields $\bfVN  = \bfV(\Omega)$), we let  $\bfW= \bfV(\Omega)\cap C^\infty_c(\Omega)^d  $ and $\Pi_N = {\rm Id}$. Then all the assumptions of this section are satisfied.
\end{remark}
\begin{remark}
 We provide in Section \ref{sec:taylorhood} the construction of $\Pi_N$ in the case of the lowest degree Taylor-Hood finite element. We show in Theorem \ref{thm:hyppinok} that it satisfies the assumptions given in this section, for a regular family of meshes in the standard sense.
\end{remark}

\subsection{Space estimate}

\begin{lemma}\label{lem:estfinite}  
Under Assumptions \eqref{hyp:T-Omega}, \eqref{hyp:f-u0}, \eqref{hyp:timediscfinite}, \eqref{hyp:X}, \eqref{hyp:M}, for any $N\ge 1$,
the following relation holds for $ n \in \llbracket 0,N-1 \rrbracket$:
\begin{multline}\label{eq:local_u_estsemidis}
  \frac 1 {2 \delta\!t_N} \left(\Vert\bfu^{\nalgo+1}_\nstep\Vert_{L^2(\Omega)^d}^2 - \Vert\bfu_\nstep^{\nalgo}\Vert_{L^2(\Omega)^d}^2\right) + \frac{\delta\!t_N}{2} \left(\Vert \nabla p_\nstep^{n+1}\Vert_{L^2(\Omega)^d}^2 - \Vert \nabla p_\nstep^{n}\Vert_{L^2(\Omega)^d}^2\right)  \\+   \frac 1 {2 \delta\!t_N} \Vert\tilde\bfu_\nstep^{\nalgo+1} - \bfu_\nstep^{\nalgo}\Vert_{L^2(\Omega)^d}^2 
   + \Vert  \tilde\bfu_N^{\nalgo+1}\Vert_{H^1_0(\Omega)^d}^2 = \langle \bff_N^{n+1} , \tilde \bfu_\nstep^{\nalgo+1} \rangle.
\end{multline}
Consequently, there exists $\ctel{cste:est}$ depending only on $|\Omega|$, on the $L^2(\Omega)^d$-norm of $\bfu_0$ and on the $L^2(0,T;H^{-1}(\Omega)^d)$-norm of $\bff$ such that the functions $\tilde \bfu_\nstep$ and $\bfu_\nstep$ defined by \eqref{eqdef:fullfunctions-finite} satisfy:
\begin{equation} \label{eq:tildeuu_estfinite} 
\|\tilde \bfu_\nstep\|_{L^2(0,T;H^1_0(\Omega)^d)} \leq \cter{cste:est} \quad \mbox{and } \quad
\|\bfu_\nstep\|_{L^{\infty}(0,T;L^{2}(\Omega)^{d})} 
 \leq \cter{cste:est}.
\end{equation}
\begin{equation}\label{eq:difftildeuu_estfinite}
\|  \tilde \bfu_\nstep \|_{L^\infty(0,T;L^2(\Omega)^d)}
\le \cter{cste:est} 
\quad \mbox{and } \quad
\| \bfu_\nstep - \tilde \bfu_\nstep \|_{L^2(0,T;L^2(\Omega)^d)}
\le \cter{cste:est} \sqrt{\delta\!t_N}.
\end{equation}
\end{lemma}

\begin{proof}
We take  $\tilde \bfu_N^{n+1} $  in \eqref{eq:finitepre}  and we obtain  for $ n \in \llbracket 0,N-1 \rrbracket$
\begin{multline}
 \label{eq:estim_semid_pred}
 \frac 1 {2\deltat_{\! N}} \Vert \tilde \bfu_\nstep^{\nalgo+1}\Vert_{L^2(\Omega)^d}^2 -\frac 1 {2\deltat_{\! N}}  \Vert \bfu_\nstep^\nalgo\Vert_{L^2(\Omega)^d}^2  +  \frac 1 {2\deltat_{\! N}} \Vert \tilde \bfu_\nstep^{\nalgo+1}- \bfu_\nstep^\nalgo\Vert_{L^2(\Omega)^d}^2   \\ + ( \nabla p_\nstep^\nalgo, \tilde \bfu_\nstep^{\nalgo+1} ) + \Vert   \tilde \bfu_\nstep^{\nalgo+1}\Vert_{H^1_0(\Omega)^d}^2  = \langle\bff_{\! N}^{n+1} , \tilde \bfu_\nstep^{\nalgo+1}\rangle. 
\end{multline}
Squaring the relation $ \bfu_N^{n+1} + \delta\!t_N \nabla p_N^{n+1} = \tilde \bfu_N^{n+1} + \delta\!t_N \nabla p_N^{n}$, integrating over $\Omega$, 
and owing to 
$\bfu_N^{n+1} \in \bfVN   $, we get that for $ n \in \llbracket 0,N-1 \rrbracket$,
\begin{multline*}
  \frac 1 {2 \deltat_{\! N}}
  \Vert \bfu_\nstep^{\nalgo+1}\Vert_{L^2(\Omega)^d}^2 + \frac{\deltat_{\! N}}{2} \| \nabla p_\nstep^{n+1}\|_{L^2(\Omega)^d}^2  = \frac 1 {2 \deltat_{\! N}}   \|\tilde\bfu_\nstep^{\nalgo+1}\|_{L^2(\Omega)^d}^2 \\ + \frac{\deltat_{\! N}}{2}  \| \nabla p_\nstep^{n}\|_{L^2(\Omega)^d}^2 + ( \tilde\bfu_\nstep^{\nalgo+1}, \nabla p_\nstep^\nalgo ).
\end{multline*}
Summing the latter relation with \eqref{eq:estim_semid_pred} yields \eqref{eq:local_u_estsemidis} for $n \in \llbracket 0,N-1 \rrbracket$.
We then get Relations \eqref{eq:tildeuu_estfinite} by summing over the time steps,  using the Cauchy-Schwarz and Poincar\'e inequalities.
\end{proof}

\begin{lemma}\label{lem:weakcv}
Under Assumptions \eqref{hyp:T-Omega}, \eqref{hyp:f-u0}, let $(\delta\!t_N,\bfX_{\!N},M_N)_{N\ge 1}$ be a sequence such that \eqref{hyp:timediscfinite}, \eqref{hyp:X}, \eqref{hyp:M} and \eqref{hyp:Mlim} hold. Then
 there exists $\bfu\in L^\infty(0,T;L^2(\Omega)^d) \cap L^2(0,T;H_0^1(\Omega)^d\cap\bfV(\Omega))$ and a subsequence again denoted by $(\delta\!t_N,\bfX_{\!N},M_N)_{N\ge 1}$ such that, letting $(\tilde \bfu_N,\bfu_N)_{N\ge 1}$ be defined by Lemma \ref{lem:exist-finite}, then $(\tilde \bfu_\nstep)_{\nstep \ge 1}$ weakly converges to $\bfu$ in $L^2(0,T;H^1_0(\Omega)^d)$ and  $( \bfu_\nstep)_{\nstep \ge 1}$ weakly-$\star$ converges to $\bfu$ in $L^\infty(0,T;L^2(\Omega)^d)$.
\end{lemma}
\begin{proof}
Owing to \eqref{eq:tildeuu_estfinite}, we get the existence of $\tilde\bfu\in L^2(0,T;H^1_0(\Omega)^d)$ such that, up to the extraction of a subsequence,  the sequence $(\tilde \bfu_\nstep)_{\nstep \ge 1}$  weakly converges to $\tilde\bfu$ in $L^2(0,T;H^1_0(\Omega)^d)$, and of $\bfu\in L^\infty(0,T;L^2(\Omega)^d)$ such that  $(\bfu_\nstep)_{\nstep \ge 1}$ weakly converges to $\bfu$ in $L^\infty(0,T;L^2(\Omega)^d)$ for the weak star topology.
Let $\xi \in H^1(\Omega) \cap L_0^2(\Omega) $ and $\varphi\in C^\infty_c(]0,T[)$ be given.

We notice that \eqref{eq:zerodiv} yields
\[
\int_0^T(\bfu_\nstep, \nabla \mathcal{P}_{M_N } \xi ) \varphi(t)\dt= 0,~\text{for any}~N \ge 1.
\]
Using the convergence of  the sequence $(\mathcal{P}_{M_N } \xi)_{N \ge 1}$ in $H^1(\Omega)$, the weak convergence of the sequence $( \bfu_N)_{N \ge 1}$ in $L^2((0,T) \times \Omega)^d$ and
passing to the limit in the previous identity gives
$$
\int_0^T( \bfu , \nabla\xi )\varphi(t)\dt=0,~\text{for any}~\xi \in H^1(\Omega).
$$
We then obtain that $\bfu \in L^2(0,T;H_0^1(\Omega)^d\cap\bfV(\Omega))$. Using \eqref{eq:difftildeuu_estfinite}, we get that $\tilde\bfu = \bfu$, which concludes the proof. 
\end{proof}

\subsection{Time estimates}\label{sec:timetrans}
The weak convergence property given by Lemma \ref{lem:weakcv} is not sufficient for  passing to the limit in the scheme, owing to the presence of the nonlinear convection term.
  Hence we need some stronger compactness property on one of the subsequences $(\bfu_\nstep)_{\nstep \ge 1}$ or $(\tilde \bfu_\nstep)_{\nstep \ge 1}$,
obtained through the application of Theorem \ref{thm:compactness1} below. We therefore need to prove an estimate on the time translates which fulfills  \eqref{eq:hyptimetranslate}. This is done by first proving  an estimate for the time translates with respect to a semi-norm which plays the role of a discrete $H^{-1}$ seminorm. This seminorm, denoted by $| \cdot |_{\ast,N}$, makes the pressure gradient in \eqref{eq:finitecor2} vanish by considering test functions $\bfv \in \bfE_N\subset \bfVN$.
It is defined  for any $\bfw\in L^2(\Omega)^d$ by
\begin{equation}
 \label{etoile-unN}
  |\bfw|_{\ast,N} = \sup\{ ( \bfw , \bfv),~ \bfv \in  \bfE_N,~ \Vert \bfv \Vert_{\bfE} \le 1\}.
\end{equation}
Recall that $\bfE_N$ and $\Vert \bfv \Vert_{\bfE}$ are defined by \eqref{eq:defbfE}-\eqref{eq:defnormeE} and that the above definition remains meaningful even if $\bfE_N = \{0\}$.
\begin{lemma}[A first estimate on the time translates]\label{lem:transsemidis}
Under Assumptions \eqref{hyp:T-Omega}, \eqref{hyp:f-u0}, \eqref{hyp:timediscfinite}, \eqref{hyp:X}, \eqref{hyp:M}, 
 there exists $\ctel{cste:esttrans1}$ only depending on $|\Omega|$, $\Vert\bfu_0\Vert_{L^2(\Omega)^d}$ and $\Vert\bff\Vert_{L^2(0,T;H^{-1}(\Omega)^d)}$ such that for any $N \ge 1$ and for any $\tau \in (0,T)$, 
\begin{equation*}
\int_0^{T-\tau} |\tilde  \bfu_\nstep(t+\tau) -  \tilde \bfu_\nstep(t) |_{\ast,N}^2 \dt \le \cter{cste:esttrans1} \tau (\tau +\deltat_{\! N}).
\end{equation*}
\end{lemma}
\begin{remark}\label{rem:ttonesthon}
 The proof of Lemma \ref{lem:transsemidis} relies on the $L^2(0,T;H^1_0(\Omega)^d)$ estimates \eqref{eq:tildeuu_estfinite}-\eqref{eq:difftildeuu_estfinite} proved in Lemma \ref{lem:estfinite}.
\end{remark}

\begin{proof}
Let $ N \ge 2$ and $\tau \in (0,T)$ (for $N=1$ the quantity we have to estimate is zero).  
Let $(\chi^n_{N,\tau})_{n \in \llbracket 1,N-1 \rrbracket}$  be the family of measurable functions defined for $n \in \llbracket 1,N-1 \rrbracket$ and $t \in \xR$ by $ \chi_{N,\tau}^n(t) =  \mathds 1_{ (t_N^n-\tau,t^n_N]}(t) $, then 
\begin{equation}
 \label{eqdef:chi-N-taufinite}
 \tilde \bfu_N(t+\tau)  -\tilde \bfu_N(t)  = \sum_{n=1}^{N-1} \chi_{N,\tau}^n(t) (\tilde \bfu_N^{n+1} - \tilde \bfu_N^n), ~\text{for any}~ t \in (0,T-\tau).
\end{equation}
Hence, owing to \eqref{eq:finitepre},  we have for any $\bfv \in \bfX_{\!N}  $ and for any $t \in (0,T-\tau)$ the following identity
\begin{multline*}
(\tilde \bfu_N(t+\tau) - \tilde \bfu_N(t),\bfv) = - \deltat_{\! N} \!\! \sum_{n=1}^{N-1} \!\!   \chi_{N,\tau}^n(t)a( \tilde \bfu_N^{n+1},\bfv)  \\  - \deltat_{\! N} \sum_{n=1}^{N-1} \!\!   \chi_{N,\tau}^n(t) b(\tilde \bfu_N^n,\tilde \bfu_N^{n+1},\bfv) 
  - \deltat_{\! N} \sum_{n=1}^{N-1} \!\!   \chi_{N,\tau}^n(t) (\nabla (2p_N^n - p_N^{n-1}),\bfv) \\ + \deltat_{\! N} \sum_{n=1}^{N-1} \!\!   \chi_{N,\tau}^n(t) \langle \bff_{\! N}^{n+1},\bfv\rangle.
\end{multline*}
Let $\bfv  \in\bfE_N$, and define $A(t)= \displaystyle\bigl(  \tilde\bfu_N(t+\tau) - \tilde \bfu_N(t) , \bfv\bigr)$. 
Using the previous identity we obtain
\[
A(t)= A_{d}(t) +  A_{c}(t) + A_{p}(t) + A_{\bff}(t),
\]
with
\[
A_{d}(t) = -\sum_{n=1}^{N-1}\chi_{N,\tau}^n(t)\deltat_{\! N} \int_\Omega \gradi \tilde{\bfu}_N^{n+1} : \gradi \bfv \dx,
\]\[
A_{c}(t) =- \sum_{n=1}^{N-1}\chi_{N,\tau}^n(t)\deltat_{\! N} b(\tilde \bfu_N^n,\tilde \bfu_N^{n+1},\bfv),
\]\[
A_{p}(t) = \sum_{n=1}^{N-1}    \chi_{N,\tau}^n(t) \deltat_{\! N} \int_\Omega  (2 p_N^{n} - p_N^{n-1}) \dive \bfv \dx,
\]\[
A_{\bff}(t) =  \sum_{n=1}^{N-1}    \chi_{N,\tau}^n(t)\deltat_{\! N} \langle \bff_{\! N}^{n+1}  ,\bfv\rangle.
\]
Using \eqref{eq:defnormeE} we have
\begin{equation}\label{eq:est_Adfinite}
A_{d}(t) \le 
 \| \bfv \|_{\bfE} \sum_{n=1}^{N-1}   \chi_{N,\tau}^n(t) \deltat_{\! N} \| \tilde{\bfu}_N^{n+1}  \|_{H_0^1(\Omega)^d}.
\end{equation}
Using the identity \eqref{eq:identityconv} and  the estimates \eqref{eq:tildeuu_estfinite}-\eqref{eq:difftildeuu_estfinite}
we have
\begin{multline}\label{eq:est_Acfinite}
A_{c}(t)  
= - \sum_{n=1}^{N-1}\chi_{N,\tau}^n(t)\deltat_{\! N} (( \tilde \bfu_N^n \cdot \gradi)  \tilde \bfu_N^{n+1} , \bfv ) \\
 - \frac{1}{2} \sum_{n=1}^{N-1}\chi_{N,\tau}^n(t)\deltat_{\! N} ( \dive \tilde \bfu_N^n \tilde \bfu_N^{n+1} , \bfv ) 
\\
\le  \sum_{n=1}^{N-1}\chi_{N,\tau}^n(t)\deltat_{\! N} \| \tilde \bfu_N^n \|_{L^2(\Omega)^d}  \| \tilde \bfu_N^{n+1} \|_{H_0^1(\Omega)^d} \| \bfv \|_{L^\infty(\Omega)^d} \\
+ \frac{1}{2}  \sum_{n=1}^{N-1}\chi_{N,\tau}^n(t)\deltat_{\! N} \| \dive \tilde \bfu_N^n \|_{L^2(\Omega)^d}  \| \tilde \bfu_N^{n+1} \|_{L^2(\Omega)^d} \| \bfv \|_{L^\infty(\Omega)^d} \\
\le \cter{cste:est}  \sum_{n=1}^{N-1}\chi_{N,\tau}^n(t)\deltat_{\! N}   \| \tilde \bfu_N^{n+1} \|_{H_0^1(\Omega)^d} \| \bfv \|_{L^\infty(\Omega)^d} \\
+ \frac{1}{2}\cter{cste:est}   \sum_{n=1}^{N-1}\chi_{N,\tau}^n(t)\deltat_{\! N} \|  \tilde \bfu_N^n \|_{H_0^1(\Omega)^d}   \| \bfv \|_{L^\infty(\Omega)^d} \\
\le \cter{cste:est} \sum_{n=1}^{N-1}   \chi_{N,\tau}^n(t)  \deltat_{\! N} (  \|  \tilde{\bfu}_N^{n+1} \|_{H_0^1(\Omega)^d} +  \|  \tilde{\bfu}_N^{n} \|_{H_0^1(\Omega)^d}) \| \bfv \|_{L^\infty(\Omega)^d} \\
\le \cter{cste:est}   \sum_{n=1}^{N-1}   \chi_{N,\tau}^n(t)  \deltat_{\! N}  (  \|  \tilde{\bfu}_N^{n+1} \|_{H_0^1(\Omega)^d} +  \|  \tilde{\bfu}_N^{n} \|_{H_0^1(\Omega)^d}) \| \bfv \|_{\bfE}.
\end{multline}
We get from $\bfE_N\subset\bfVN$ that $A_p(t) =0$. 
Next, we note that
\begin{equation}\label{eq:est_Affinite}
A_{\bff}(t) \le 
C_{\text{sob}} \| \bfv \|_{\bfE} \sum_{n=1}^{N-1}    \chi_{N,\tau}^n(t)\deltat_{\! N} \|\bff_{\! N}^{n+1}\|_{H^{-1}(\Omega)^d}.
\end{equation}
where $C_{\text{sob}}$ is such that $\| \bfv \|_{L^2(\Omega)^d} \le C_{\text{sob}}\| \bfv \|_{H_0^1(\Omega)^d}$ for any $\bfv \in H_0^1(\Omega)^d$.
Summing Equations \eqref{eq:est_Adfinite}, \eqref{eq:est_Acfinite}, \eqref{eq:est_Affinite}, we obtain
$$
A(t) \le C  \| \bfv \|_{\bfE}  \sum_{n=1}^{N-1}   \chi_{N,\tau}^n(t) \deltat_{\! N}(  \| \tilde \bfu_N^{n+1} \|_{H_0^1(\Omega)^d} + \| \tilde \bfu_N^{n} \|_{H_0^1(\Omega)^d}+ \| \bff_{\! N}^{n+1} \|_{H^{-1}(\Omega)^d})
$$
where $ C = \frac 3 2\cter{cste:est} +C_{\text{sob}} +1$.
This implies 
$$
|\tilde  \bfu_N(t+\tau) -  \tilde \bfu_N(t) |_{\ast,N} \le C  \sum_{n=1}^{N-1}  \chi_{N,\tau}^n(t) \deltat_{\! N} ( \| \tilde \bfu_N^{n+1} \|_{H_0^1(\Omega)^d}+\| \tilde \bfu_N^{n} \|_{H_0^1(\Omega)^d}  + \| \bff_{\! N}^{n+1} \|_{H^{-1}(\Omega)^d}).
$$
Since $\sum_{n=1}^{N-1}  \chi_{N,\tau}^n(t) \deltat_{\! N} \le \tau +\deltat_{\! N} $ for any $t \in (0,T-\tau)$ we then obtain
$$
|\tilde  \bfu_N(t+\tau) -  \tilde \bfu_N(t) |_{\ast,N}^2 \le 3 C^2 (\tau +\deltat_{\! N})  \sum_{n=1}^{N-1} \chi_{N,\tau}^n(t)\deltat_{\! N} ( \| \tilde \bfu_N^{n+1} \|_{H_0^1(\Omega)^d}^2+\| \tilde \bfu_N^{n} \|^2_{H_0^1(\Omega)^d}  + \| \bff_{\! N}^{n+1} \|_{H^{-1}(\Omega)^d}^2).
$$
Noting that  $\int_0^{T-\tau} \chi_{N,\tau}^n(t) \dt \le \tau $ for any $n \in \llbracket 1,N-1 \rrbracket$ yields
\begin{multline*}
\int_0^{T-\tau} |\tilde  \bfu_N(t+\tau) -  \tilde \bfu_N(t) |_{\ast,N}^2 \dt  \\  \le 3 C^2 (\tau +\deltat_{\! N}) \sum_{n=1}^{N-1}  \deltat_{\! N} ( \| \tilde \bfu_N^{n+1} \|_{H_0^1(\Omega)^d}^2+ \| \tilde \bfu_N^{n} \|^2_{H_0^1(\Omega)^d} + \| \bff_{\! N}^{n+1} \|_{H^{-1}(\Omega)^d}^2) \int_0^{T-\tau}  \chi_{N,\tau}^n(t) \dt \\
\le   3 C^2 (\tau +\deltat_{\! N}) \tau ( 2 \| \tilde \bfu_N \|_{L^2(0,T:H_0^1(\Omega)^d)}^2 + \| \bff \|_{L^2(0,T;H^{-1}(\Omega)^d)}^2) 
\le \cter{cste:esttrans1} \tau (\tau +\deltat_{\! N})
\end{multline*}
which gives the expected result.

\end{proof}

\begin{remark}
Note that the property   $\| \bfv \|_{L^\infty(\Omega)^d} \le \| \bfv \|_{\bfE}$ for any $\bfv \in \bfE_N$ is only used in \eqref{eq:est_Acfinite} to estimate the discrete convective term. 
\end{remark}

 Note that Lemma \ref{lem:transsemidis} is only the first step for proving an estimate on the time translates of the predicted velocity in the $L^2(L^2)$ norm. The next steps are based on the Lions-like result given below, for which we need the following result.
 
 \begin{lemma}\label{lem:contprojN}
Under Assumptions \eqref{hyp:T-Omega}, \eqref{hyp:X}, \eqref{hyp:M}, \eqref{hyp:PNexist}, \eqref{hyp:PNconv}, \eqref{hyp:Mlim},
let $(\bfv_N)_{N \ge 1}$ be a sequence of functions  of $L^2(\Omega)^d$
such that $(\bfv_N)_{N\ge 1}$ converges to $\bfv $ in $L^2(\Omega)^d$. 
Then the sequence $(\mathcal{P}_{\bfV_{\! N}(\Omega)}\bfv_N)_{N \ge 1}$ converges to $\mathcal{P}_{\bfV(\Omega)} \bfv$ in $L^2(\Omega)^d$.
\end{lemma}
\begin{proof}
Using the fact that $(\bfv_N)_{N\ge 1}$  is bounded in $L^2(\Omega)^d$ we obtain that the sequence $(\mathcal{P}_{\bfV_{\! N}(\Omega)}\bfv_N)_{N \ge 1}$ is bounded in $L^2(\Omega)^d$. 
Hence there exists a subsequence still denoted by $(\mathcal{P}_{\bfV_{\! N}(\Omega)}\bfv_N)_{N \ge 1}$ that  converges to a function $\tilde \bfv$ weakly in $L^2(\Omega)^d$. 
Let $\xi \in H^1(\Omega) \cap L_0^2(\Omega)$.
Using the fact that for any $N \ge 1$ we have $\mathcal{P}_{\bfV_{\! N}(\Omega)}\bfv_N \in \bfVN   $ and $\mathcal{P}_{M_N } \xi \in M_N $ we obtain
\[
(\mathcal{P}_{\bfV_{\! N}(\Omega)}\bfv_N , \nabla  \mathcal{P}_{M_N } \xi ) = 0,~\text{for any}~N \ge 1.
\]
Using the weak convergence of  the sequence $(\mathcal{P}_{\bfV_{\! N}(\Omega)}\bfv_N)_{N \ge 1}$ in $L^2(\Omega)^d$  and the strong convergence of the sequence $(\mathcal{P}_{M_N } \xi)_{N \ge 1}$ in $H^1(\Omega)$ and
passing to the limit in the previous identity gives
$$
(\tilde \bfv , \nabla\xi )=0,~\text{for any}~\xi \in H^1(\Omega).
$$
We then obtain that $\tilde \bfv \in \bfV(\Omega)$. Let $ \bfvarphi \in  \bfW$ be given. Using the fact that $\Pi_N \bfvarphi \in \bfVN   $ for any $ N\ge 1$, we obtain
$$
  (\bfv_N , \Pi_{N} \bfvarphi )  =( \mathcal{P}_{\bfV_{\! N}(\Omega)}\bfv_N , \Pi_{N} \bfvarphi ),~\text{for any}~N \ge 1.
$$
Using the convergence to $\bfvarphi$ of  the sequence $(\Pi_N \bfvarphi)_{N \ge 1}$ in $L^2(\Omega)^d$, the weak convergence of the sequence $(\mathcal{P}_{\bfV_{\! N}(\Omega)}\bfv_N)_{N \ge 1}$ in $L^2(\Omega)^d$,
passing to the limit in the previous identity and using the fact that $( \bfv , \bfvarphi) = (\mathcal{P}_{\bfV}  \bfv , \bfvarphi)$ for any $\bfvarphi \in \bfW$ give
$$
(\mathcal{P}_{\bfV}  \bfv , \bfvarphi) = (\tilde \bfv , \bfvarphi )~\text{for any}~\bfvarphi \in \bfW.
$$
Using the fact that 
$\bfW$  is dense  in  $\bfV(\Omega)$ for the $L^2(\Omega)^d$ norm we obtain that $ \tilde \bfv = \mathcal{P}_{\bfV} \bfv$ and that the sequence $(\mathcal{P}_{\bfV_{\! N}(\Omega)}\bfv_N)_{N \ge 1}$ converges to $\mathcal{P}_{\bfV(\Omega)} \bfv$ weakly in $L^2(\Omega)^d$. We can write
$$
\| \mathcal{P}_{\bfV_{\! N}(\Omega)}\bfv_N \|_{L^2(\Omega)^d}^2
= ( \bfv_N , \mathcal{P}_{\bfV_{\! N}(\Omega)}\bfv_N ),~\text{for any}~N \ge 1.
$$
Using the convergence of the sequence $(\bfv_N)_{N \ge 1}$ to $\bfv$ in $L^2(\Omega)^d$ and the weak convergence of the sequence $(\mathcal{P}_{\bfV_{\! N}(\Omega)}\bfv_N)_{N \ge 1}$ to $\mathcal{P}_{\bfV(\Omega)} \bfv$ in $L^2(\Omega)^d$ we obtain
$$
\lim_{\Nti} \| \mathcal{P}_{\bfV_{\! N}(\Omega)}\bfv_N \|_{L^2(\Omega)^d}^2 = ( \bfv , \mathcal{P}_{\bfV(\Omega)} \bfv ) = \| \mathcal{P}_{\bfV(\Omega)} \bfv \|_{L^2(\Omega)^d}^2.
$$
The weak convergence of the sequence $(\mathcal{P}_{\bfV_{\! N}(\Omega)}\bfv_N)_{N \ge 1}$ to $\mathcal{P}_{\bfV(\Omega)} \bfv$ in $L^2(\Omega)^d$  and the convergence of the sequence $( \| \mathcal{P}_{\bfV_{\! N}(\Omega)}\bfv_N \|_{L^2(\Omega)^d})_{N \ge 1}$ to $ \| \mathcal{P}_{\bfV(\Omega)} \bfv \|_{L^2(\Omega)^d}$ give the expected result.
\end{proof}

\begin{remark}
 Note that, in the preceding proof, we only use the assumptions weaker than  \eqref{hyp:PNconv} and \eqref{hyp:Mlim}  that $\bfW$ is a dense subset of $ \bfV(\Omega)$  for the norm of $L^2(\Omega)^d$ and that, for any $N \ge 1$ there exists a mapping $\Pi_N : \bfW \to  \bfV_N  $ such that for any $\bfvarphi\in  \bfW $ such that the sequence $(\Pi_N\bfvarphi )_{N \ge 1}$ converges to  $\bfvarphi$ in $L^2(\Omega)^d$. 
\end{remark}
 
\begin{lemma}[Lions-like]\label{lem:lionsfinite}
Under Assumptions \eqref{hyp:T-Omega}, \eqref{hyp:X}, \eqref{hyp:M},  \eqref{hyp:PNexist} \eqref{hyp:PNconv}, \eqref{hyp:Mlim}, we have
\begin{multline}\label{eq:f}
 \forall \varepsilon >0,\ \exists C_\varepsilon >0,\ \exists N_\epsilon \ge 1,\ \forall N \ge N_\epsilon,\ \forall \bfw \in \bfX_{\!N}  \\
\| \mathcal{P}_{\bfVN   } \bfw \|_{L^2(\Omega)^d}  \le \varepsilon \Vert \bfw \Vert_{H_0^1(\Omega)^d} + C_\varepsilon |\bfw|_{\ast,N}.    
\end{multline}
\end{lemma}
\begin{proof}
Let us assume that \eqref{eq:f} does not hold. This means that there exists $\varepsilon >0$ such that
\begin{multline}\label{eq:notf}
 \forall  C\ge 1,\ \forall M \ge 1,\ \exists N \ge M,\ \exists \bfw \in \bfX_{\!N}  \\
 \| \mathcal{P}_{\bfVN   } \bfw \|_{L^2(\Omega)^d}  > \varepsilon \Vert \bfw \Vert_{H_0^1(\Omega)^d} + C |\bfw|_{\ast,N}.    
\end{multline}
Let us set $\nu(0) = 0$, and let us build the infinite set $\mathcal{I} = \{\nu(n),n\in\mathbb{N}^\star\}\subset \mathbb{N}^\star$ and the sequence $(\bfw_N)_{N\in \mathcal{I}}$ by induction. For $n\in\mathbb{N}^\star$, we select $C= n$  and $M= \nu(n-1) + 1$ in \eqref{eq:notf}. We get the existence of a given $N\ge M$ and of a given $\bfw \in \bfX_{N}(\Omega)$ such that 
\[
 \| \mathcal{P}_{\bfVN   } \bfw \|_{L^2(\Omega)^d}   > \varepsilon \Vert \bfw \Vert_{H_0^1(\Omega)^d} + n |\bfw|_{\ast,N}. 
\]
We then define $\nu(n) = N$ and $\bfw_N= \bfw/ \| \mathcal{P}_{\bfVN } \bfw \|_{L^2(\Omega)^d}  $. Let us denote $\mathcal{I}= \{ \nu(n)~\text{such that}~n \in \mathbb{N}^\star\} \subset \mathbb{N}^\star.$ We denote by $\nu^{-1}$ the reciprocal function of $\nu:\mathbb{N}^\star\to \mathcal{I}$. 
Since the mapping $\nu$ is strictly increasing, we then indeed get that $\mathcal{I}$ is infinite.
We can then write, for any $N\in \mathcal{I}$,
\begin{equation*}
    1 > \varepsilon \Vert \bfw_N \Vert_{H_0^1(\Omega)^d} + \nu^{-1}(N) |\bfw_N|_{\ast,N},~\text{for any}~N\in \mathcal{I}.
\end{equation*}
It then follows from the latter inequality that the sequence $( \bfw_N)_{N\in \mathcal{I}}$ is bounded in $H^1_0(\Omega)^d$ and that $|\bfw_N|_{\ast,N} \to 0$ as $N\to\infty$ with $N\in \mathcal{I}$. 
Hence there exists an infinite set $\mathcal{J}\subset  \mathcal{I}$ such that $ (\bfw_N)_{N\in \mathcal{J}} $ converges in $L^2(\Omega)^d$ to a function $\bfw \in H_0^1(\Omega)^d$ while $|\bfw_N|_{\ast,N} \to 0$ as $N\to\infty$ with $N\in \mathcal{J}$. We notice that $ \| \mathcal{P}_{\bfVN   }\bfw_N \|_{L^2(\Omega)^d} = 1 $ for all  $N\in \mathcal{J}$.

Using Lemma \ref{lem:contprojN} we have
$ \mathcal{P}_{\bfVN   } \bfw_{N} \to \mathcal{P}_{\bfV(\Omega)}  \bfw$ in  $L^2(\Omega)^d$ as $N\to\infty$ with $N\in \mathcal{J}$, which therefore implies $\|\mathcal{P}_{\bfV}  \bfw \|_{L^2(\Omega)^d}=1$.

Let $ \bfvarphi \in  \bfW$ be given. 
For any $N\in \mathcal{J}$, by definition of $|\bfw_N|_{\ast,N}$, we have 
\[
   (\bfw_N , \Pi_N \bfvarphi) \le |\bfw_N|_{\ast,N} \Vert  \Pi_N\bfvarphi \Vert_{\bfE}.
\]
Using the fact that, for any $N\in \mathcal{J}$, $\Pi_N \bfvarphi \in \bfVN   $, we then obtain
\[
    ( \bfw_N ,\Pi_N \bfvarphi)  = ( \mathcal{P}_{\bfVN   } \bfw_N ,\Pi_N \bfvarphi)  \le |\bfw_N|_{\ast,N} \Vert  \Pi_N\bfvarphi \Vert_{\bfE}.
\]
Since $ \Vert  \Pi_N\bfvarphi \Vert_{\bfE}$ remains bounded by assumption \eqref{hyp:PNconv}, letting $N\to\infty$ with $N\in \mathcal{J}$ in this inequality yields that 
\[
  ( \mathcal{P}_{\bfV(\Omega)}\bfw , \bfvarphi ) = 0,~\text{for any}~ \bfvarphi \in \bfW.
\]
Using the fact that 
$\bfW$  is dense  in  $\bfV(\Omega)$ for the $L^2(\Omega)^d$ norm 
we can therefore let $\bfvarphi \to \mathcal{P}_{\bfV(\Omega)}\bfw$ in $L^2(\Omega)^d$. This yields
$$
\| \mathcal{P}_{\bfV(\Omega)} \bfw \|_{L^2(\Omega)^d}^2 =  0,
$$  
which contradicts $\Vert \mathcal{P}_{\bfV(\Omega)} \bfw \Vert_{L^2(\Omega)^d} = 1$.
\end{proof}
\begin{remark}
Note that, in the preceding proof, we only used the assumptions weaker than   \eqref{hyp:PNconv} and \eqref{hyp:Mlim}  that $\bfW$ is a dense subset of $ \bfV(\Omega)$  for the norm of $L^2(\Omega)^d$ and that,  for any $N \ge 1$ there exists a mapping $\Pi_N : \bfW \to  \bfE_N  $ such that for any $\bfvarphi\in  \bfW $ such that the sequence $(\Pi_N\bfvarphi )_{N \ge 1}$ converges to  $\bfvarphi$ in $L^2(\Omega)^d$ and the sequence $( \|\Pi_N \bfvarphi \|_{\bfE})_{N \ge 1}$ is bounded. 
\end{remark}

Our aim is now to use Lemma \ref{lem:transsemidis} on the time translates of $(\tilde \bfu_N)_{N \ge 1}$ for the $L^2(|\cdot|_{\ast,N})$ semi-norm and \eqref{eq:f} in the above lemma, in order to obtain an estimate on the time translates for the $L^2(0,T;L^2(\Omega)^d)$ norm, as stated by the next lemma.
\begin{lemma}[$L^2$ estimate on the time translates]\label{lem:trans-utildefinite}
Under Assumptions \eqref{hyp:T-Omega}, \eqref{hyp:f-u0}, \eqref{hyp:timediscfinite}, \eqref{hyp:X}, \eqref{hyp:M}, \eqref{hyp:PNexist}, \eqref{hyp:PNconv}, \eqref{hyp:Mlim}, the sequence $(\tilde \bfu_\nstep)_{N \ge 1}$ satisfies
\begin{equation}
\label{eq:trans-utildefinite}
 \int_0^{T-\tau}\!\!\!\!\!\!\!\!\Vert \tilde \bfu_\nstep(t+\tau) - \tilde \bfu_\nstep(t) \Vert_{L^2(\Omega)^d}^2 \dt \to 0 \mbox{ as } \tau \to 0, \mbox{ uniformly with respect to }N, 
\end{equation}
and is therefore relatively compact in $L^2(0,T;L^2(\Omega)^d)$.
\end{lemma}
\begin{proof}
Let us show that $A_\nstep(\tau) \to 0$ as $\tau \to 0$ uniformly with respect to $N$, where we define for any $\tau\in(0,T)$
\[
 A_\nstep(\tau) := \int_0^{T-\tau} \Vert \tilde \bfu_\nstep(t+\tau) - \tilde \bfu_\nstep(t) \Vert_2^2 \dt.
\]
For any $\bfw\in L^2(\Omega)^d$ and $\bfv\in \bfVN$, we have
\[
 \bfw = \bfw - \mathcal{P}_{\bfVN   }(\bfw) + \mathcal{P}_{\bfVN   }(\bfw)  =\bfw-\bfv - \mathcal{P}_{\bfVN   }(\bfw-\bfv) + \mathcal{P}_{\bfVN   }(\bfw).
\]
Using the fact that $\mathcal{P}_{\bfVN   }$ is an orthogonal projection, we have
\[
 \Vert \bfw-\bfv - \mathcal{P}_{\bfVN   }(\bfw-\bfv) \Vert_{L^2(\Omega)^d} \le \Vert \bfw-\bfv\Vert_{L^2(\Omega)^d},
\]
which leads to
\[
 \Vert \bfw \Vert_{L^2(\Omega)^d} \le \Vert \bfw-\bfv\Vert_{L^2(\Omega)^d} + \Vert \mathcal{P}_{\bfVN   }(\bfw)\Vert_{L^2(\Omega)^d},
\]
giving
\[
 \forall \bfw\in L^2(\Omega)^d,\ \forall \bfv\in \bfVN,\ \Vert \bfw \Vert_{L^2(\Omega)^d}^2 \le 2\Vert \bfw-\bfv\Vert_{L^2(\Omega)^d}^2 + 2\Vert \mathcal{P}_{\bfVN   }(\bfw)\Vert_{L^2(\Omega)^d}^2.
\]
We apply the preceding inequality, letting $\bfw = \tilde  \bfu_\nstep(t+\tau) - \tilde  \bfu_\nstep(t)$ and $\bfv =  \bfu_\nstep(t+\tau) -  \bfu_\nstep(t)$ and we integrate on $t\in(0,T-\tau)$. Setting
\[ 
 B_\nstep(\tau) =  \int_0^{T-\tau} \Vert (\tilde \bfu_\nstep -\bfu_\nstep)(t+\tau) - (\tilde \bfu_\nstep-\bfu_\nstep)(t) \Vert_2^2 \dt
\]
we obtain
\begin{equation}\label{eq:majantau}
A_N(\tau)  \le 2 B_\nstep(\tau) +  2 \int_0^{T-\tau} \| \mathcal{P}_{\bfVN   }(\tilde \bfu_\nstep(t+\tau) -  \tilde \bfu_\nstep(t))\|_{L^2(\Omega)^d}^2 \dt.
\end{equation}

Let $\zeta > 0$ be given.

From \eqref{eq:difftildeuu_estfinite}, we have $B_\nstep(\tau)\le 4\cter{cste:est}^2 \deltat_N$. Let $N_B\ge 1$ be such that $ 4\cter{cste:est}^2 \deltat_N \le \zeta$ for all $N\ge N_B$. Since $B_\nstep(\tau) \to 0$ as $\tau \to 0$ for $N=1,\ldots,N_B$, we can choose $\tau_B$ such that, for any $0\le \tau\le \tau_B$ and $N=1,\ldots,N_B$, we have $B_\nstep(\tau) \le \zeta$.

This yields $B_\nstep(\tau)\le \zeta$ for all $N\ge 1$ and any $0\le \tau<\tau_B$.

Our aim is now to use \eqref{eq:f} in Lemma \ref{lem:lionsfinite}, which implies 
\begin{multline*}
 \forall \varepsilon >0,\ \exists C_\varepsilon >0,\ \exists N_\epsilon \ge 1,\ \forall N \ge N_\epsilon,\ \forall \bfw \in \bfX_{\!N}  \\
 \| \mathcal{P}_{\bfVN   }(\bfw)\|_{L^2(\Omega)^d}^2 \le 2\varepsilon^2 \Vert \bfw \Vert_{H_0^1(\Omega)^d}^2 + 2C_\varepsilon^2 |\bfw|_{\ast,N}^2,
\end{multline*}
letting $\bfw =   \tilde  \bfu_\nstep(t+\tau) -  \tilde \bfu_\nstep(t) $.

From \eqref{eq:tildeuu_estfinite}, we have 
\begin{equation}\label{eq:majbrutal}
\int_0^{T-\tau} \Vert \tilde  \bfu_\nstep(t+\tau) -  \tilde \bfu_\nstep(t) \Vert_{H^1_0(\Omega)^d}^2 \dt \le 4 \cter{cste:est}^2.
\end{equation}

We then select $\varepsilon$ such that
\begin{equation}\label{eq:defvareps}
 2\varepsilon^2 4 \cter{cste:est}^2 = \zeta.
\end{equation}
Then there exists $C_\zeta  >0$ and $N_\zeta$ such that for any $N \ge N_\zeta$ and for any $\tau \in (0,T)$ and for any $t \in (0,T-\tau)$,
\begin{equation*}
 \| \mathcal{P}_{\bfVN   }(\tilde \bfu_\nstep(t+\tau) -  \tilde \bfu_\nstep(t))\|_{L^2(\Omega)^d}^2 \le 2\varepsilon^2 \Vert \tilde  \bfu_\nstep(t+\tau) -  \tilde \bfu_\nstep(t) \Vert_{H_0^1(\Omega)^d}^2 + C_\zeta^2 |\tilde  \bfu_\nstep(t+\tau) -  \tilde \bfu_\nstep(t) |_{\ast,N}^2.
\end{equation*}
Integrating the previous relation provides, using \eqref{eq:majbrutal} and \eqref{eq:defvareps},
\begin{multline*}
 \int_0^{T-\tau} \| \mathcal{P}_{\bfVN   }(\tilde \bfu_\nstep(t+\tau) -  \tilde \bfu_\nstep(t))\|_{L^2(\Omega)^d}^2 \dt \le \zeta \\ + 2C_\zeta^2 \int_0^{T-\tau} |\tilde  \bfu_\nstep(t+\tau) -  \tilde \bfu_\nstep(t) |^2_{\ast,N} \dt.
\end{multline*}
Thus, owing to Lemma \ref{lem:transsemidis},  we have for any $N \ge N_\zeta$ and 
for any $\tau \in (0,T)$
\begin{equation*}
\int_0^{T-\tau} \| \mathcal{P}_{\bfVN   }(\tilde \bfu_\nstep(t+\tau) -  \tilde \bfu_\nstep(t))\|_{L^2(\Omega)^d}^2 \dt 
\le\zeta  +  2C_\zeta^2 \cter{cste:esttrans1} \tau (\tau+\deltat_{\! N}),
\end{equation*}
and therefore, for any $ N \ge N_\zeta$ and for any $\tau \in (0,T)$, using \eqref{eq:majantau}
$$
A_\nstep(\tau) \le 2 B_\nstep(\tau) +  2\zeta + 4 C_\zeta^2 \cter{cste:esttrans1} \tau(\tau+\deltat_{\! N}).
$$
We now choose $\tau_C>0$, such that, for any $0\le\tau\le \tau_C$, $4 C_\zeta^2 \cter{cste:esttrans1} \tau(\tau+T) \le \zeta$.

We then get, for  any $ N \ge N_\zeta$ and $0\le \tau \le \min(\tau_B,\tau_C)$,
$$
A_\nstep(\tau) \le 2 \zeta +  2\zeta + \zeta = 5\zeta.
$$
Since  $A_\nstep(\tau) \to 0$ as $\tau \to 0$ for $N=1,\ldots,N_\zeta$, we can choose $\tau_D$ such that, for any $0\le \tau\le \tau_D$ and $N=1,\ldots,N_\zeta$, we have $A_\nstep(\tau) \le 5\zeta$.
\

We then obtain that $A_\nstep(\tau) \le 5 \zeta$ for any $ \tau \in [0, \min(\tau_B, \tau_C,\tau_D)] $ and $N\ge 1$.
 The proof of \eqref{eq:trans-utildefinite} is thus complete.
\end{proof}

\subsection{Convergence of the projection scheme to a weak solution}\label{sec:convweaksol}

The convergence proof given below relies on the following result which is a consequence of the Kolmogorov theorem, as noticed in \cite[Corollaire 4.41]{gallouet:cel-01196782}. 
\begin{theorem}\label{thm:compactness1}
Let $1 \le p < + \infty$.
 Let $B$ be a Banach space and let $X$ be a  Banach space compactly embedded in $B$. Let $(u_N)_{N \ge 1}$ be a sequence of $L^p(0,T;B)$ satisfying the following conditions:
\begin{enumerate}
\item The sequence $(u_N)_{N \ge 1}$ is bounded in $L^p(0,T;B)$.
\item The sequence $( u_N )_{N \ge 1}$ is bounded in $L^1(0,T;X)$.
\item  The sequence $(u_N)_{N \ge 1}$ satifies 
 \begin{equation}\label{eq:hyptimetranslate}
    \int_0^{T-\tau} \Vert u_{N}(t+\tau) - u_{N}(t) \Vert_{B}^p \dt \to 0 \mbox{ as } \tau \to 0,
\end{equation}
uniformly with respect to $N \ge 1$.
\end{enumerate}
Then the sequence $(u_N)_{N \ge 1}$ is relatively compact in $L^p(0,T;B)$.
\end{theorem}

Let us now detail the way this result is used.
By Lemma \ref{lem:weakcv}, up to a subsequence, the sequence 
$(\tilde \bfu_N)_{N \ge 1}$ weakly converges to some limit $\bar \bfu$ in $L^2(0,T;L^2(\Omega)^d)$,  and owing to \eqref{eq:difftildeuu_estfinite}, so does the sequence $(\bfu_N)_{N \ge 1}$.
There remains to check that $\bar \bfu$ is a weak solution to \eqref{pb:cont} in the sense of Definition \ref{def:weaksol}, by passing to the limit in the scheme with suitable test functions. This is possible using the compactness theorem \ref{thm:compactness1} and the time translate estimate lemma \ref{lem:trans-utildefinite}.
\begin{theorem}\label{thm:cvscheme}
Under Assumptions \eqref{hyp:T-Omega}, \eqref{hyp:f-u0}, let $(\delta\!t_N,\bfX_{\!N},M_N,\Pi_N)_{N\ge 1}$ be a sequence such that \eqref{hyp:timediscfinite}, \eqref{hyp:X}, \eqref{hyp:M}, \eqref{hyp:PNexist}, \eqref{hyp:PNconv}, \eqref{hyp:Mlim} hold. Then there exists a subsequence samely denoted,  and a function $\bar u\in L^2(0,T;H_0^1(\Omega)^d\cap\bfV(\Omega))$ $\cap$ $L^\infty(0,T;L^2(\Omega)^d)$,  such that the corresponding sequences  $(\tilde \bfu_N)_{N \ge 1}$ and $(\bfu_N)_{N \ge 1}$ defined by the incremental projection scheme \eqref{eq:semidiscretefinite}  converge to $\bar \bfu$ strongly in $L^2(0,T;L^2(\Omega)^d)$.
 Moreover the function $\bar \bfu$ is a weak solution to \eqref{pb:cont} in the sense of Definition \ref{def:weaksol}.
 \end{theorem}
 \begin{proof}
 We proceed in several steps.
 \begin{itemize}
 \item \textit{Step 1: compactness and convergence in $L^2$.}
Using Lemma \ref{lem:trans-utildefinite}, we can apply Theorem \ref{thm:compactness1}. We get the existence of $\bar \bfu$ in $L^2(0,T;L^2(\Omega)^d)$ and of a subsequence of  $(\delta\!t_N,\bfX_{\!N},M_N,\Pi_N)_{N\ge 1}$, still denoted  $(\delta\!t_N,\bfX_{\!N},M_N,\Pi_N)_{N\ge 1}$, such that the corresponding sequences $(\bfu_\nstep)_{\nstep \ge 1}$ and $(\tilde \bfu_\nstep)_{\nstep \ge 1}$ converge to  $\bar \bfu$ in $L^2(0,T;L^2(\Omega)^d)$. Owing to Lemma \ref{lem:weakcv},
we get that $\bar \bfu\in L^\infty(0,T;L^2(\Omega)^d) \cap L^2(0,T;H_0^1(\Omega)^d\cap\bfV(\Omega))$.
 \item \textit{Step 2: convergence towards a weak solution.}
There remains to show that $\bar \bfu$ is a weak solution in the sense of Definition \ref{def:weaksol} and in particular that $\bar \bfu $ satisfies \eqref{weaksol}.
Let $ \psi \in C_c^\infty ( [0,T))$ and  $\bfvarphi \in \bfW$. Let $\bfv $ be the function defined by $\bfv(t,\bfx)= \psi(t) \bfvarphi(\bfx)$ for any $(t,\bfx)\in[0,T)\times\Omega$. Let $N\ge 1$.
Let $(\bfv_N^n)_{n \in \llbracket 0,N \rrbracket}$ be the sequence of functions of $\bfX_{\!N} \cap \bfVN  $ defined by
$
\bfv_N^n  = \psi(t_N^n) \Pi_N \bfvarphi$ for any $n \in \llbracket 0,N \rrbracket$. Let $\bfv_N : (0,T) \to H_0^1(\Omega)^d$ be defined by
\begin{equation*}
 \bfv_\nstep (t) = \sum_{n=0}^{\nstep-1} \ \Char_{(t_N^\nalgo,t_N^{n+1}]}(t) \bfv_N^{\nalgo+1}.
\end{equation*}
We remark that we have
\begin{multline}
\| \bfv_N - \bfv \|_{L^\infty(0,T;H_0^1(\Omega)^d)} \le \delta\!t_N \| \psi' \|_{L^\infty(0,T)} \| \Pi_N \bfvarphi \|_{\bfE}
\\ + \| \psi \|_{L^\infty(0,T)} \| \Pi_N \bfvarphi - \bfvarphi\|_{H_0^1(\Omega)^d},~\text{for any}~N \ge 1.
\end{multline}
Consequently we obtain using \eqref{hyp:PNconv}  that the sequence $(\bfv_N)_{N \ge 1}$ converges to $\bfv$ in $L^\infty(0,T;H_0^1(\Omega)^d)$. 
Multiplying \eqref{eq:finitepre} by $\deltat_{\! N} \bfv_N^{n+1}$, using the fact  that $(\tilde \bfu_N^{n+1},\bfv_N^{n+1})=(\bfu_N^{n+1},\bfv_N^{n+1})$ for any $n \in \llbracket0,N-1\rrbracket$ and summing over $n \in \llbracket 0,N-1 \rrbracket$ lead to
\begin{multline}\label{eq:weaksolsemidisRvarphi}
\sum_{n=0}^{N-1} (  \bfu_N^{n+1} -  \bfu_N^n , \bfv_N^{n+1} )  + \delta\!t_N \sum_{n=0}^{N-1} b(\tilde \bfu_N^n,\tilde \bfu_N^{n+1},\bfv_N^{n+1})\\ +\int_{0}^T \int_\Omega \gradi \tilde \bfu_N : \gradi \bfv_N \dx \dt = \int_{0}^T \langle \bff , \bfv_N \rangle \dt. 
\end{multline}
Using the fact that $\bfv_N^{N}=0$ in $\Omega$ the first term of the left hand side reads
\begin{multline*}
\sum_{n=0}^{N-1} ( \bfu_N^{n+1} -  \bfu_N^n , \bfv_N^{n+1} ) 
= - \sum_{n=0}^{N-1}( \bfu_N^n,\bfv_N^{n+1}-\bfv_N^n) - ( \bfu_N^0 , \bfv_N^0 ) \\
= - \int_0^T  \psi'(t) (  \bfu_N(t) ,  \Pi_N \bfvarphi) \dt- \psi(0) (\mathcal{P}_{\bfVN} \bfu_0 , \Pi_N \bfvarphi)
\\
= - \int_0^T  \psi'(t) (  \bfu_N(t) ,  \Pi_N \bfvarphi) \dt- \psi(0) ( \bfu_0 , \Pi_N \bfvarphi).
\end{multline*}
Since the sequence $( \bfu_N)_{N \ge 1}$ converges to $\bar \bfu$ in $L^2(0,T;L^2(\Omega)^d)$ and the sequence $(\Pi_N\bfvarphi)_{N \ge 1}$ converges to $ \bfvarphi$ in $L^2(\Omega)^d$, we can then write
\begin{equation}\label{eq:convweakstepfinite1}
\lim_ {\Nti}\sum_{n=0}^{N-1} ( \bfu_N^{n+1}- \bfu_N^n , \bfv_N^{n+1} )  = - \int_0^T ( \bar \bfu, \partial_t \bfv ) \dt - (\bfu_0 , \bfv(0,\cdot) ).
\end{equation}
Using the fact that the initial predicted velocity is zero we have
\begin{multline*}
\delta\!t_N \sum_{n=0}^{N-1} b(\tilde \bfu_N^n,\tilde \bfu_N^{n+1},\bfv_N^{n+1}) = \delta\!t_N \sum_{n=1}^{N-1} b(\tilde \bfu_N^n,\tilde \bfu_N^{n+1}, \bfv_N^{n+1}) \\
 =  \frac{1}{2}\int_{\delta\!t_N}^T (( \tilde \bfu_N(t-\delta\!t_N) \cdot \gradi) \tilde \bfu_N(t) , \bfv_N(t)) \dt \\ - \frac{1}{2}\int_{\delta\!t_N}^T (( \tilde \bfu_N(t-\delta\!t_N) \cdot \gradi) \bfv_N(t) , \tilde \bfu_N(t)) \dt.
\end{multline*}
The first term of the discrete convective term can be expressed by
\begin{multline*}
\int_{\delta\!t_N}^T (( \tilde \bfu_N(t-\delta\!t_N) \cdot \gradi) \tilde \bfu_N(t) , \bfv_N(t)) \dt =\int_{0}^T (( \tilde \bfu_N(t) \cdot \gradi) \tilde \bfu_N(t) , \bfv_N(t)) \dt \\ 
-\int_{0}^{\delta\!t_N} (( \tilde \bfu_N(t) \cdot \gradi) \tilde \bfu_N(t) , \bfv_N(t)) \dt \\
+ \int_{\delta\!t_N}^T ((( \tilde \bfu_N(t-\delta\!t_N) - \tilde \bfu_N(t)) \cdot \gradi) \tilde \bfu_N(t) , \bfv_N(t)) \dt.
\end{multline*}
Using Lemma \ref{lem:estfinite}, we get the following bound for 
the second term of the right hand-side  of the previous identity:
$$
\int_{0}^{\delta\!t_N} (( \tilde \bfu_N(t) \cdot \gradi) \tilde \bfu_N(t) , \bfv_N(t)) \dt \le \cter{cste:est}^2 \| \psi \|_{L^\infty(0,T)} \| \Pi_N \bfvarphi \|_{L^\infty(\Omega)^d} \sqrt{\delta\!t_N}.
$$
Using the fact that the sequence $( \| \Pi_N \bfvarphi \|_{L^\infty(\Omega)^d} )_{N \ge 1}$ is bounded we obtain that this term tends to zero as $N$ tends to infinity. 
Using Lemma \ref{lem:estfinite}, a bound for the third term of the right hand-side of the previous identity is given by
\begin{multline*}
\int_{\delta\!t_N}^T ((( \tilde \bfu_N(t-\delta\!t_N) - \tilde \bfu_N(t)) \cdot \gradi) \tilde \bfu_N(t) , \bfv_N(t)) \dt \\ \le \cter{cste:est} \| \psi \|_{L^\infty(0,T)} \| \Pi_N \bfvarphi \|_{L^\infty(\Omega)^d} \Big( \int_0^{T-\delta\!t_N} \| \tilde \bfu_N(t+\delta\!t_N) - \tilde \bfu_N(t) \|^2_{L^2(\Omega)^d} \Big)^{1/2}.
\end{multline*}
Using  Lemma \ref{lem:trans-utildefinite} and the fact that   $( \| \Pi_N \bfvarphi \|_{L^\infty(\Omega)^d} )_{N \ge 1}$ is bounded we obtain that the preceding term tends to zero as $N$ tends to infinity.
Using the convergence of the sequence $(\tilde \bfu_N)_{N \ge 1}$ in $L^2(0,T;L^4(\Omega)^d)$, the weak convergence of the sequence $(\tilde \bfu_N)_{N \ge 1}$ in $L^2(0,T;H_0^1(\Omega)^d)$ and  the convergence of the sequence $ ( \bfv_N)_{N \ge 1}$ in $L^\infty(0,T;L^4(\Omega)^d)$ we obtain 
$$
\lim_{\Nti} \int_{0}^T (( \tilde \bfu_N(t) \cdot \gradi) \tilde \bfu_N(t) , \bfv_N(t)) \dt  = \int_0^T ( (\bar \bfu \cdot \gradi) \bar \bfu ,  \bfv ) \dt.
$$
The second term of the discrete convective term is such that
\begin{multline*}
\int_{\delta\!t_N}^T (( \tilde \bfu_N(t-\delta\!t_N) \cdot \gradi) \bfv_N(t) , \tilde \bfu_N(t)) \dt =\\
\int_{0}^T (( \tilde \bfu_N(t) \cdot \gradi) \tilde \bfu_N(t) , \bfv_N(t)) \dt 
-\int_{0}^{\delta\!t_N} (( \tilde \bfu_N(t) \cdot \gradi)  \bfv_N(t) , \tilde \bfu_N(t)) \dt \\
+\int_{\delta\!t_N}^T ((( \tilde \bfu_N(t-\delta\!t_N) - \tilde \bfu_N(t)) \cdot \gradi)  \bfv_N(t) , \tilde \bfu_N(t)) \dt.
\end{multline*}
A bound for the second term of the previous identity is obtained by
$$
\int_{0}^{\delta\!t_N} (( \tilde \bfu_N(t) \cdot \gradi)  \bfv_N(t) , \tilde \bfu_N(t)) \dt \le \| \psi \|_{L^\infty(0,T)} \| \Pi_N \bfvarphi \|_{H^1_0(\Omega)^d} \int_0^{\delta\!t_N} \| \tilde \bfu_N(t) \|_{L^4(\Omega)^d}^2 \dt.
$$
Using the fact that the sequence $(\tilde \bfu_N)_{N \ge 1}$ is bounded in $L^2(0,T;L^4(\Omega)^d)$ and the fact the sequence $( \| \Pi_N \bfvarphi \|_{H_0^1(\Omega)^d} )_{N \ge 1}$ is bounded we obtain that this term tends to zero as $N$ tends to infinity. 
The third term of the previous equation can be bounded as follows
\begin{multline*}
\int_{\delta\!t_N}^T ((( \tilde \bfu_N(t-\delta\!t_N) - \tilde \bfu_N(t)) \cdot \gradi)  \bfv_N(t) , \tilde  \bfu_N(t)) \dt \\ \le \cter{cste:est} \| \psi \|_{L^\infty(0,T)} \| \Pi_N \bfvarphi \|_{H_0^1(\Omega)^d} \Big( \int_0^{T-\delta\!t_N} \| \tilde \bfu_N(t+\delta\!t_N) - \tilde \bfu_N(t) \|^2_{L^4(\Omega)^d} \Big)^{1/2} \\
\times \Big( \int_0^{T-\delta\!t_N} \| \tilde \bfu_N(t) \|^2_{L^4(\Omega)^d} \Big)^{1/2}  \\
\le \cter{cste:est} \| \psi \|_{L^\infty(0,T)} \| \Pi_N \bfvarphi \|_{H_0^1(\Omega)^d} \Big( \int_0^{T-\delta\!t_N} \| \tilde \bfu_N(t+\delta\!t_N) - \tilde \bfu_N(t) \|^2_{L^2(\Omega)^d} \Big)^{1/8} \\
\times  \Big( \int_0^{T-\delta\!t_N} \| \tilde \bfu_N(t+\delta\!t_N) - \tilde \bfu_N(t) \|^2_{L^6(\Omega)^d} \Big)^{3/8} \Big( \int_0^{T-\delta\!t_N} \| \tilde \bfu_N(t) \|^2_{L^4(\Omega)^d} \Big)^{1/2}
\end{multline*}
Using the estimates on the translation of the sequence $(\tilde \bfu_N)_{N \ge 1}$ given in Lemma \ref{lem:trans-utildefinite}, the fact that the sequence $(\tilde \bfu_N)_{N \ge 1}$ is bounded in $L^2(0,T;L^6(\Omega)^d)$  and the fact that   $( \| \Pi_N \bfvarphi \|_{H_0^1(\Omega)^d} )_{N \ge 1}$ is bounded we obtain that this term tends to zero as $N$ tends to infinity.
Using the convergence of the sequence $(\tilde \bfu_N)_{N \ge 1}$ in $L^2(0,T;L^4(\Omega)^d)$ and  the convergence of the sequence $ ( \bfv_N)_{N \ge 1}$ in $L^\infty(0,T;H_0^1(\Omega)^d)$ we obtain 
$$
\lim_{\Nti} \int_{0}^T (( \tilde \bfu_N(t) \cdot \gradi)  \bfv_N(t) , \tilde \bfu_N(t)) \dt  = \int_0^T ( (\bar \bfu \cdot \gradi)  \bfv ,  \bar \bfu ) \dt.
$$
We then obtain 
\begin{equation}\label{eq:convweakstepfinite2}
\lim_{\Nti}  \delta\!t_N \sum_{n=0}^{N-1} b(\tilde \bfu_N^n,\tilde \bfu_N^{n+1},\bfv_N^{n+1})  = \int_0^T ( (\bar \bfu \cdot \gradi) \bar \bfu ,  \bfv ) \dt.
\end{equation}
Using the weak convergence of the sequence $(\gradi \tilde \bfu_N)_{N \ge 1}$ in $L^2(0,T;L^2(\Omega)^d)$ and the convergence of the sequence $ ( \gradi \bfv_N)_{N \ge 1}$ in $L^2(0,T;L^2(\Omega)^d)$ we obtain
\begin{equation}\label{eq:convweakstepfinite3}
\lim_{\Nti}\int_{0}^{T} \int_\Omega \gradi \tilde \bfu_N : \gradi \bfv_N \dx \dt = \int_0^T \int_\Omega \gradi \bar \bfu : \gradi \bfv \dx \dt.
\end{equation}
Using  the convergence of the sequence $ (  \bfv_N)_{N \ge 1}$ in $L^2(0,T;L^2(\Omega)^d)$  we obtain
\begin{equation}\label{eq:convweakstepfinite4}
\lim_{\Nti} \int_{0}^T \langle \bff , \bfv_N \rangle \dt = \int_0^T \langle \bff , \bfv \rangle \dt.
\end{equation}
Using \eqref{eq:convweakstepfinite1}-\eqref{eq:convweakstepfinite4} and passing to the limit in \eqref{eq:weaksolsemidisRvarphi} gives the weak sense \eqref{weaksol} with $\bfv$ as test function. We get that \eqref{weaksol} holds for all test functions by density of the set of linear combinations of such functions.
\end{itemize}
\end{proof}

\section{Lowest-order Taylor-Hood finite element approximation}\label{sec:taylorhood}

In this section, we study the case of the lowest-order Taylor-Hood element. This pair is of particular interest, since it is the lowest-order conforming finite element method which is available on simplicial meshes (triangles in 2D, tetrahedra in 3D). It is such that the discrete velocity space $\bfX_{\!N}$ contains continuous piecewise quadratic functions vanishing at the boundary, and the discrete pressure space $M_{\!N}$ contains continuous piecewise affine functions.
\\
Our aim is to show that it is possible to construct an operator $\Pi_N$ satisfying  the convergence assumptions \eqref{hyp:PNexist}-\eqref{hyp:Mlim} done in Section \ref{sec:cvhyp}, under the only standard regularity hypothesis on the meshes that the ratio between exterior and interior diameters of the triangles or tetrahedra is uniformly bounded.  We emphasize that the hypotheses which must be verified in  Section \ref{sec:cvhyp} are not equivalent to the inf-sup condition necessary for schemes coupling the velocities and the pressure (as in \cite{guer2007faedo,fer2018proof}). Notice that the verification of this inf-sup stability is not easy in the case of the Taylor-Hood approximation (see \cite{Bercovier1979ErrorEF}, \cite{Boland1985StableAS}, \cite{Stenberg1984AnalysisOM}, \cite{Brezzi1991StabilityOH}, \cite{Boffi1994STABILITYOH}). 
Indeed, this stability condition does not hold on general 2D and 3D simplicial meshes, only assuming the standard regularity hypothesis as recalled above. A few pathological cases can be listed in 2D for meshes containing 1, 2 or 3 triangles and in 3D for meshes containing a series of tetrahedra whose edges are all on the boundary, but these meshes are not relevant in the case of a convergence study.
~
\medskip
~\\
But there also exists standard situations which make the inf-sup condition fail, while considering a convergence study based on a sequence of 3D meshes whose size tends to zero, the regularity property being checked. Indeed, Figure \protect{\ref{fig:3Dmeshes}} shows the case where the mesh includes at least one tetrahedron (denoted $K_0$) whose all the edges are located on the boundary, with one vertex  (denoted by $\bfy_0$) only belonging to $K_0$. Then the inf-sup condition cannot hold. Indeed, for any $\bfv\in \bfX_{\!N}$, $\bfv = 0$ on $K_0$, and for  $q\in M_N$ such that $q(\bfy_0) = 1+c$ and  $q(\bfy_i) = c$ for any vertex $\bfy_i\neq\bfy_0$, $c$ being adjusted for yielding null mean value for $q$, then $(\gradi q,\bfv) = 0$.

\begin{figure}[!ht]
  \begin{center}
 \begin{tikzpicture}[scale=1]
  
  \fill[ red!5, opacity=1.] (1,1)--(6,1)--(5,5)--(1,1);
  \fill[ blue!30, opacity=1.] (5.5,1)--(5.9,1.4)--(5.7,1.5)--(5.5,1);
  
  \draw[-, very thick, black](1,1)--(6,1)--(5,5)--(3,6)--(1,1);
  \draw[-, very thick, black](6,1)--(3,6);
  \draw[dotted, very thick, black](1,1)--(5,5);
 \draw[-, blue](5.5,1)--(6,1)--(5.9,1.4)--(5.7,1.5)--(5.5,1);
  \draw[-, blue](6,1)--(5.7,1.5);
  \draw[dotted, blue](5.5,1)--(5.9,1.4);
  
  \draw[stealth-] (5.75,1.1)--(6.5,2);\path(6.7,2.2) node[black]{$K_0$};
  
   \draw[fill, red!30, opacity=1.](5.5,1) circle (0.05);
   \draw[fill, red!30, opacity=1.](6,1) circle (0.05);
   \draw[fill, red!30, opacity=1.](5.9,1.4) circle (0.05);
   \draw[fill, red!30, opacity=1.](5.7,1.5) circle (0.05);
   
%

\path(6.3,1) node[black]{$\bfy_0$};
 
  \end{tikzpicture}
  \caption{An example leading to inf-sup condition failure  for  the lowest-order Taylor-Hood finite element in 3D. }
  \protect{\label{fig:3Dmeshes}}
  \end{center}
\end{figure}
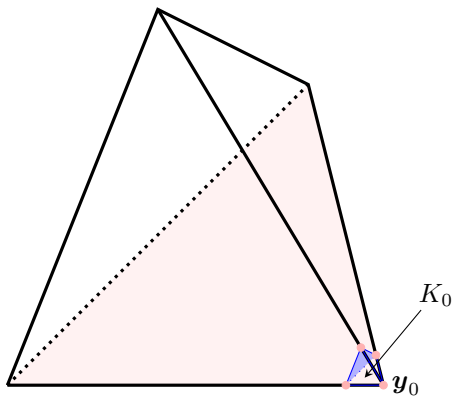

~
\medskip
~\\
Sufficient conditions on the mesh (which then prevent from the cases listed above) are then requested for the proof of the existence of Fortin operators, that are defined as bounded interpolation
operators preserving the discrete divergence of a function. The existence of a Fortin operator is then equivalent to the inf-sup condition (indeed, in the example of Figure \ref{fig:3Dmeshes}, no operator can keep the discrete divergence in the case of a function whose support is included in the corner tetrahedron).  We refer to \cite{Chen2014ASC}, \cite{Mardal2013AUS}, \cite{Falk2008AFO} for the construction of a Fortin operator in the two dimensional case. The extension of the 2D face-bubble Fortin operator to the 3D case has been done in  \cite{Diening2021FortinOF}. It relies on edge-bubble functions, which are modified at the boundary in order that the interpolation vanishes on $\partial \Omega$, under the constraint that each simplex as at least one interior node (this  eliminates the above cases of inf-sup condition failure). 
~
\medskip
~\\
In this paper, instead of constructing a complete Fortin operator, we focus on the only properties needed for the operator $\Pi_N$ involved in the convergence of the projection scheme.
We recall below the definition of the Lagrange interpolator $\Pi_L$  and its approximation properties (which are elementary compared to that of the Scott-Zhang interpolator, and lead to very short proofs). We then adapt and simplify the divergence correcting linear operator $\Pi_D$ provided by  \cite{Diening2021FortinOF} to the case of compact support functions. This enables us to define the operator $\Pi_N$ from $\Pi_L$ and $\Pi_D$ in the case of any simplicial mesh (including all the cases where the inf-sup condition fails).

\subsection{Regular simplicial mesh}

Let us assume that the computational domain $\Omega$ is polygonal ($d = 2$) or polyhedral ($d = 3$).
We define a simplex in dimension $d\ge 1$ as the interior of the convex hull of a given set of $d+1$ points (called its vertices) which are not all contained in the same hyperplane (a simplex is a triangle if $d=2$ and a tetrahedron if $d=3$).
~\\
We consider a regular simplicial mesh $\mathcal{T}$ of $\Omega$ in the usual sense of the finite element literature \cite{ciarlet}, as depicted in Figure \ref{fig:mesh}, which means that:

\begin{itemize}
 \item The family containing all the vertices of the elements of the mesh
is denoted by $(\bfy_i)_{i\in\mathcal{N}}$. The set ${\mathcal N}$ is partitioned into ${\mathcal N} ={\mathcal N}_{\rm int}\cup {\mathcal N}_{\rm ext}$, where $(\bfy_i)_{i \in {\mathcal N}_{\rm ext}}$ is the set of  exterior nodes and  $(\bfy_i)_{i \in {\mathcal N}_{\rm int}}$ is the set of  interior nodes.
\item For any $K\in \mathcal{T}$, the family of the vertices of $K$ is denoted by $(\bfy_i)_{i\in \mathcal{N}_K}$ with  $\mathcal{N}_K\subset \mathcal{N}$ contains $d+1$ elements. 
For any $K \in \mathcal{T}$, we denote by $|K|$ the measure in $\mathbb{R}^d$ of $K$, by $h_K$ the diameter of $K$ and by $\rho_K$ the diameter of the largest ball included in $K$. 
Then, we define the mesh size $h_\mathcal{T}$ and the mesh regularity $\theta_{\mathcal{T}}$ by 
$$ h_{\mathcal{T}}=\max_{K \in \mathcal{T}} h_K, \quad \quad \theta_{\mathcal{T}} = \max_{K \in \mathcal{T}} \frac{ h_K}{\rho_K}. $$ 
\item For any $K\in \mathcal{T}$, the set of the faces of $K$ is denoted by $\mathcal{F}_K$. For any $F\in \mathcal{F}_K$, the family of the vertices of the simplex $F$ is denoted by $(\bfy_i)_{i\in \mathcal{N}_F}$ (then  $\mathcal{N}_F\subset \mathcal{N}_K$ contains $d$ elements).
\item The set  $\mathcal{E}$  of the edges of the elements contains all pairs $\{i,j\}$ such that there exists $K\in\mathcal{T}$ with $\{i,j\}\subset \mathcal{N}_K$. For any $\sigma = \{i,j\}\in \mathcal{E}$, we denote by $\bfy_{\sigma}  = \frac 1 2(\bfy_i + \bfy_j)$. 
For all $K\in \mathcal{T}$, we denote by $\mathcal{E}_K$ the subset of $\mathcal{E}$ containing all the edges of $K$. 
\end{itemize}


For $i \in \mathcal{N}$ (resp.  $\{i,j\}\in \mathcal{E}$), we define $\mathcal{T}_i$ (resp. $\mathcal{T}_{ij}$) as the set of all $K\in \mathcal{T}$ such that $i\in {\mathcal N}_K$ (resp. $\{i,j\}\in \mathcal{E}_K$) and we denote by 
\[
 \omega_i = \bigcup_{K\in \mathcal{T}_i} \overline{K}\hbox{ and }
 \omega_{ij} = \bigcup_{K\in \mathcal{T}_{ij}} \overline{K}.
\]

\begin{figure}[!ht]
  \begin{center}
  \begin{tikzpicture}[scale=.8]
  
   \draw[-, blue!20, pattern=north west lines, pattern color=blue!30, opacity=0.8] (3,1)--(5,1)--(6,2.4)--(5,3.8)--(3,3.8)--(2,2.4); \path(4.4,2.6) node[black]{$\bfy_i$};\path(3,1.8) node[blue]{$\omega_i$};

  \draw[-, very thick, blue](3,1)--(5,1);
  \draw[-, very thick, blue](5,1)--(7,1);
  \draw[-, very thick, blue](3,1)--(2,2.4);
  \draw[-, very thick, blue](3,1)--(4,2.4);
  \draw[-, very thick, blue](5,1)--(4,2.4);
  \draw[-, very thick, blue](5,1)--(6,2.4);
  \draw[-, very thick, blue](7,1)--(6,2.4);
  \draw[-, very thick, blue](7,1)--(8,2.4);

  \draw[-, very thick, blue](2,2.4)--(4,2.4);
  \draw[-, very thick, blue](4,2.4)--(6,2.4);
  \draw[-, very thick, blue](6,2.4)--(8,2.4);
  \draw[-, very thick, blue](2,2.4)--(3,3.8);
  \draw[-, very thick, blue](4,2.4)--(3,3.8);
  \draw[-, very thick, blue](4,2.4)--(5,3.8);
  \draw[-, very thick, blue](6,2.4)--(5,3.8);
  \draw[-, very thick, blue](6,2.4)--(7,3.8);
  \draw[-, very thick, blue](8,2.4)--(7,3.8);

  \draw[-, very thick, blue](3,3.8)--(5,3.8);
  \draw[-, very thick, blue](5,3.8)--(7,3.8);
  \draw[-, very thick, blue](4,5.2)--(3,3.8);
  \draw[-, very thick, blue](4,5.2)--(5,3.8);
  \draw[-, very thick, blue](6,5.2)--(5,3.8);
  \draw[-, very thick, blue](6,5.2)--(7,3.8);

  \draw[-, very thick, blue](4,5.2)--(6,5.2);

   \draw[fill, red!30, opacity=1.](3,1) circle (0.1);
  \draw[fill, red!30, opacity=1.](5,1) circle (0.1);
  \draw[fill, red!30, opacity=1.](7,1) circle (0.1);
  \draw[fill, red!30, opacity=1.](2,2.4) circle (0.1);
  \draw[fill, red!30, opacity=1.](4,2.4) circle (0.1);
  \draw[fill, red!30, opacity=1.](6,2.4) circle (0.1);
  \draw[fill, red!30, opacity=1.](8,2.4) circle (0.1);
  \draw[fill, red!30, opacity=1.](3,3.8) circle (0.1);
  \draw[fill, red!30, opacity=1.](5,3.8) circle (0.1);
  \draw[fill, red!30, opacity=1.](7,3.8) circle (0.1);
  \draw[fill, red!30, opacity=1.](4,5.2) circle (0.1);
  \draw[fill, red!30, opacity=1.](6,5.2) circle (0.1);

   \draw[fill, black, opacity=1.](3,1) circle (0.05);
  \draw[fill, black, opacity=1.](5,1) circle (0.05);
  \draw[fill, black, opacity=1.](7,1) circle (0.05);
  \draw[fill, black, opacity=1.](2,2.4) circle (0.05);
  \draw[fill, black, opacity=1.](4,2.4) circle (0.05);
  \draw[fill, black, opacity=1.](6,2.4) circle (0.05);
  \draw[fill, black, opacity=1.](8,2.4) circle (0.05);
  \draw[fill, black, opacity=1.](3,3.8) circle (0.05);
  \draw[fill, black, opacity=1.](5,3.8) circle (0.05);
  \draw[fill, black, opacity=1.](7,3.8) circle (0.05);
  \draw[fill, black, opacity=1.](4,5.2) circle (0.05);
  \draw[fill, black, opacity=1.](6,5.2) circle (0.05);
  
   \draw[fill, black, opacity=1.](4,1) circle (0.05);
   \draw[fill, black, opacity=1.](2.5,1.7) circle (0.05);
   \draw[fill, black, opacity=1.](3.5,1.7) circle (0.05);
   \draw[fill, black, opacity=1.](6,1) circle (0.05);
   \draw[fill, black, opacity=1.](4.5,1.7) circle (0.05);
   \draw[fill, black, opacity=1.](5.5,1.7) circle (0.05);
   \draw[fill, black, opacity=1.](6.5,1.7) circle (0.05);
   \draw[fill, black, opacity=1.](7.5,1.7) circle (0.05);

   \draw[fill, black, opacity=1.](3,2.4) circle (0.05);
   \draw[fill, black, opacity=1.](2.5,3.1) circle (0.05);
   \draw[fill, black, opacity=1.](3.5,3.1) circle (0.05);
   \draw[fill, black, opacity=1.](5,2.4) circle (0.05);
   \draw[fill, black, opacity=1.](4.5,3.1) circle (0.05);
   \draw[fill, black, opacity=1.](5.5,3.1) circle (0.05);
   \draw[fill, black, opacity=1.](7,2.4) circle (0.05);
   \draw[fill, black, opacity=1.](6.5,3.1) circle (0.05);
   \draw[fill, black, opacity=1.](7.5,3.1) circle (0.05);

   \draw[fill, black, opacity=1.](4,3.8) circle (0.05);
   \draw[fill, black, opacity=1.](3.5,4.5) circle (0.05);
   \draw[fill, black, opacity=1.](6,3.8) circle (0.05);
   \draw[fill, black, opacity=1.](4.5,4.5) circle (0.05);
   \draw[fill, black, opacity=1.](5.5,4.5) circle (0.05);
   \draw[fill, black, opacity=1.](6.5,4.5) circle (0.05);

   \draw[fill, black, opacity=1.](5,5.2) circle (0.05);

  \end{tikzpicture}
  \hspace{.5cm}
  \begin{tikzpicture}[scale=.8]
  
   \draw[fill=none](5,4.685) circle (0.5);
   \draw[stealth-stealth](4.5,4.685)--(5.5,4.685);;\path(5,4.8) node[black]{$\rho_K$};
   \draw[stealth-stealth](4,5.3)--(6,5.3);\path(5,5.5) node[black]{$h_K$};\path(4.4,5) node[black]{$K$};
   \draw[-,blue!20, pattern=north east lines, pattern color=black!30, opacity=0.8] (5,3.8)--(6,5.2)--(4,5.2);
   
   \draw[-, blue!20, pattern=north west lines, pattern color=blue!30, opacity=0.8] (5,1)--(6,2.4)--(5,3.8)--(4,2.4); \path(3.6,2.6) node[black]{$\bfy_i$};\path(6.4,2.6) node[black]{$\bfy_j$};\path(5,1.8) node[blue]{$\omega_{ij}$};
   
   \path(5,2.6) node[black]{$\bfy_{\{i,j\}}$};
   
   \draw[fill, black](5,2.4) circle (0.05);
   
  \draw[-, very thick, blue](3,1)--(5,1);
  \draw[-, very thick, blue](5,1)--(7,1);
  \draw[-, very thick, blue](3,1)--(2,2.4);
  \draw[-, very thick, blue](3,1)--(4,2.4);
  \draw[-, very thick, blue](5,1)--(4,2.4);
  \draw[-, very thick, blue](5,1)--(6,2.4);
  \draw[-, very thick, blue](7,1)--(6,2.4);
  \draw[-, very thick, blue](7,1)--(8,2.4);

  \draw[-, very thick, blue](2,2.4)--(4,2.4);
  \draw[-, very thick, blue](4,2.4)--(6,2.4);
  \draw[-, very thick, blue](6,2.4)--(8,2.4);
  \draw[-, very thick, blue](2,2.4)--(3,3.8);
  \draw[-, very thick, blue](4,2.4)--(3,3.8);
  \draw[-, very thick, blue](4,2.4)--(5,3.8);
  \draw[-, very thick, blue](6,2.4)--(5,3.8);
  \draw[-, very thick, blue](6,2.4)--(7,3.8);
  \draw[-, very thick, blue](8,2.4)--(7,3.8);

  \draw[-, very thick, blue](3,3.8)--(5,3.8);
  \draw[-, very thick, blue](5,3.8)--(7,3.8);
  \draw[-, very thick, blue](4,5.2)--(3,3.8);
  \draw[-, very thick, blue](4,5.2)--(5,3.8);
  \draw[-, very thick, blue](6,5.2)--(5,3.8);
  \draw[-, very thick, blue](6,5.2)--(7,3.8);

  \draw[-, very thick, blue](4,5.2)--(6,5.2);

   \draw[fill, red!30, opacity=1.](3,1) circle (0.1);
  \draw[fill, red!30, opacity=1.](5,1) circle (0.1);
  \draw[fill, red!30, opacity=1.](7,1) circle (0.1);
  \draw[fill, red!30, opacity=1.](2,2.4) circle (0.1);
  \draw[fill, red!30, opacity=1.](4,2.4) circle (0.1);
  \draw[fill, red!30, opacity=1.](6,2.4) circle (0.1);
  \draw[fill, red!30, opacity=1.](8,2.4) circle (0.1);
  \draw[fill, red!30, opacity=1.](3,3.8) circle (0.1);
  \draw[fill, red!30, opacity=1.](5,3.8) circle (0.1);
  \draw[fill, red!30, opacity=1.](7,3.8) circle (0.1);
  \draw[fill, red!30, opacity=1.](4,5.2) circle (0.1);
  \draw[fill, red!30, opacity=1.](6,5.2) circle (0.1);

   \draw[fill, black, opacity=1.](3,1) circle (0.05);
  \draw[fill, black, opacity=1.](5,1) circle (0.05);
  \draw[fill, black, opacity=1.](7,1) circle (0.05);
  \draw[fill, black, opacity=1.](2,2.4) circle (0.05);
  \draw[fill, black, opacity=1.](4,2.4) circle (0.05);
  \draw[fill, black, opacity=1.](6,2.4) circle (0.05);
  \draw[fill, black, opacity=1.](8,2.4) circle (0.05);
  \draw[fill, black, opacity=1.](3,3.8) circle (0.05);
  \draw[fill, black, opacity=1.](5,3.8) circle (0.05);
  \draw[fill, black, opacity=1.](7,3.8) circle (0.05);
  \draw[fill, black, opacity=1.](4,5.2) circle (0.05);
  \draw[fill, black, opacity=1.](6,5.2) circle (0.05);

   \draw[fill, black, opacity=1.](4,1) circle (0.05);
   \draw[fill, black, opacity=1.](2.5,1.7) circle (0.05);
   \draw[fill, black, opacity=1.](3.5,1.7) circle (0.05);
   \draw[fill, black, opacity=1.](6,1) circle (0.05);
   \draw[fill, black, opacity=1.](4.5,1.7) circle (0.05);
   \draw[fill, black, opacity=1.](5.5,1.7) circle (0.05);
   \draw[fill, black, opacity=1.](6.5,1.7) circle (0.05);
   \draw[fill, black, opacity=1.](7.5,1.7) circle (0.05);

   \draw[fill, black, opacity=1.](3,2.4) circle (0.05);
   \draw[fill, black, opacity=1.](2.5,3.1) circle (0.05);
   \draw[fill, black, opacity=1.](3.5,3.1) circle (0.05);
   \draw[fill, black, opacity=1.](5,2.4) circle (0.05);
   \draw[fill, black, opacity=1.](4.5,3.1) circle (0.05);
   \draw[fill, black, opacity=1.](5.5,3.1) circle (0.05);
   \draw[fill, black, opacity=1.](7,2.4) circle (0.05);
   \draw[fill, black, opacity=1.](6.5,3.1) circle (0.05);
   \draw[fill, black, opacity=1.](7.5,3.1) circle (0.05);

   \draw[fill, black, opacity=1.](4,3.8) circle (0.05);
   \draw[fill, black, opacity=1.](3.5,4.5) circle (0.05);
   \draw[fill, black, opacity=1.](6,3.8) circle (0.05);
   \draw[fill, black, opacity=1.](4.5,4.5) circle (0.05);
   \draw[fill, black, opacity=1.](5.5,4.5) circle (0.05);
   \draw[fill, black, opacity=1.](6.5,4.5) circle (0.05);

   \draw[fill, black, opacity=1.](5,5.2) circle (0.05);

  \end{tikzpicture}
  
  \caption{An interior part of a 2D mesh. The degrees of freedom for the pressure are in pink, those for the components of the velocity are in  black. }\label{fig:mesh}
  \end{center}
\end{figure}
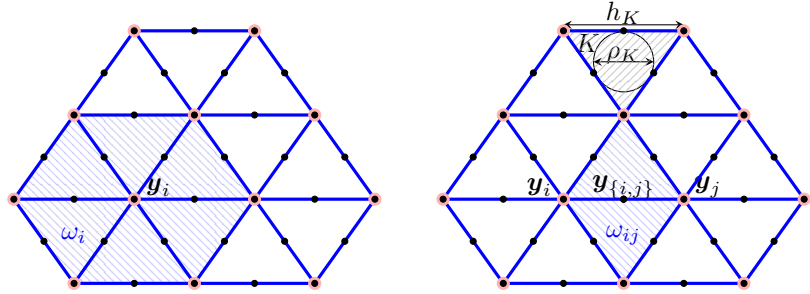

In the whole section \ref{sec:taylorhood}, we denote by $C_i$, with $i\in\mathbb{N}$, non negative real values which may depend on $d$ and on $\Omega$, and increasingly on $\theta_{\mathcal{T}}$.
The following lemma is a direct consequence of the definition of $\theta_{\mathcal{T}}$.
\begin{lemma}\label{lem:reguomegaij}
 There exists $\ctel{cst:regumesh}$ such that
 \[
  {\rm card}\mathcal{T}_{ij} \le \cter{cst:regumesh},
 \]

 \[
 \forall K\in \mathcal{T}_{ij},\  \frac 1 {\cter{cst:regumesh}} h_K^d \le |\omega_{ij}| \le \cter{cst:regumesh} h_K^d,
 \]
 \begin{equation}\label{eq:diamomegaij}
  \forall K\in \mathcal{T}_{ij},\  {\rm diam}(\omega_{ij}) \le \cter{cst:regumesh} h_K.
 \end{equation}

\end{lemma}

\subsection{Discrete spaces and bases}

 We introduce the set $\mathbb{P}^k(K)$ of the polynomials over $K \in \mathcal{T}$  
of degree less than or
equal to $k$, regardless of the space dimension.  We define
$$
\mathbb{P}^k(\mathcal{T}) = \{ p \in C(\overline{\Omega},\mathbb{R}) ~\text{such that}~p_{|K} \in \mathbb{P}^k(K)~\text{for any}~K \in \mathcal{T} \}.
$$

\medskip

We denote by $(\varphi_i)_{i \in \mathcal{N}}$ the nodal basis of $\mathbb{P}^1(\mathcal{T})$, that is, 
for any $i\in \mathcal{N}$, the element $\varphi_i$ of $\mathbb{P}^1(\mathcal{T})$ 
such that $\varphi_i(\bfy_i) = 1$ and $\varphi_i(\bfy_j) = 0$ for all $j\in \mathcal{N}\setminus\{i\}$.

 Using the properties of the family $(\varphi_i)_{i \in \mathcal{N}}$ we have 
\begin{equation}\label{eq:supportvarphi}
\overline{\{ \bfx \in \Omega~\text{such that}~\varphi_i(\bfx) \ne 0 \} } = \omega_i~\text{for any}~i \in \mathcal{N}.
\end{equation}

Then the nodal basis of $\mathbb{P}^2(\mathcal{T})$ is given by the family 
$\Big((\varphi_i(2\varphi_i-1))_{i\in\mathcal{N}},(4\varphi_i \varphi_j)_{\{i,j\}\in\mathcal{E}}\Big)$.
Denoting this family by $(\phi_i)_{i \in \mathcal{N}\cup\mathcal{E}}$,
for any $i\in \mathcal{N}\cup\mathcal{E}$, we have that $\phi_i(\bfy_i) = 1$ and $\phi_i(\bfy_j) = 0$ for all $j\in (\mathcal{N}\cup\mathcal{E})\setminus\{i\}$.

Using \eqref{eq:supportvarphi} and using \cite[Proposition 11.6]{Ern2021FiniteEI} we can state the following lemma.

\begin{lemma}
There exists $\ctel{cste:estgradvarphi}$ such that 
\begin{equation}\label{eq:estgradivarphi}
\forall K\in\mathcal{T},\  \max_{i \in \mathcal{N}_K} \| \nabla\varphi_i \|_{L^\infty(\omega_i)^d} \le \frac{\cter{cste:estgradvarphi} }{h_K}
\end{equation}
which leads that there exists $\ctel{cst:majgradp2}$ 
\begin{equation}\label{eq:estgradphiP2}
 \forall K\in\mathcal{T},\  \forall \bfx\in K, \ \forall i\in\mathcal{N}_K\cup\mathcal{E}_K,\ |\nabla\phi_i(\bfx)|\le \frac {\cter{cst:majgradp2}} {h_K}.
\end{equation}

\end{lemma}

The spaces of approximation for the velocity and the pressure are defined by
\begin{equation}\label{eq:defxnmn}
 \bfX_{\!N} = \mathbb{P}^2(\mathcal{T})^d \cap H_0^1(\Omega)^d
\quad \text{and} \quad M_N = \mathbb{P}^1(\mathcal{T}) \cap L_0^2(\Omega).
\end{equation}

\subsection{The $\mathbb{P}^2(\mathcal{T})$ Lagrange interpolator}

The Lagrange interpolator, denoted by 
$\Pi_L : C^\infty_c(\Omega)^d \to \bfX_{\!N}$, 
is defined by 
\begin{equation}\label{eq:defpun}
\forall \bfv =( v_k)_{k=1,\ldots,d} \in  C^\infty_c(\Omega)^d,\ (\Pi_L \bfv)_k = \sum_{i \in \mathcal{N}\cup\mathcal{E} }  v_k(\bfy_i)\phi_i,\ k = 1,\ldots,d.
\end{equation}

\begin{lemma}\label{lem:estP1inf}
There exists  $\ctel{cste:estP1inf}$ such that for any $\bfv \in C^\infty_c(\Omega)^d$, we have
\begin{equation}\label{eq:estP1inf}
\| \Pi_L \bfv \|_{L^\infty(\Omega)^d} \le \cter{cste:estP1inf} \| \gradi \bfv \|_{L^\infty(\Omega)^{d \times d}}.
\end{equation}
\end{lemma}
\begin{proof}
 Let $\underline{x}_1,\overline{x}_1\in\mathbb{R}$ such that, for any $\bfx =(x_1,\ldots,x_d)\in\Omega$, $\underline{x}_1\le x_1\le\overline{x}_1$. Writing, for $k=1,\ldots,d$,
 \[
   v_k(\bfy_i) = \int_{\underline{x}_1}^{(y_i)_1} \partial_1  v_k(s,((y_i)_\ell)_{\ell=2,\ldots,d}){\rm d}s,
 \]
 we get 
 \[
  | v_k(\bfy_i)|\le (\overline{x}_1-\underline{x}_1) \| \gradi \bfv \|_{L^\infty(\Omega)^{d \times d}}.
 \]
 Since $|\phi_i(\bfx)|\le 1$ for all  $\bfx\in\Omega$, and using that the maximum number of $i\in\mathcal{N}\cup\mathcal{E}$ such that  $\phi_i(\bfx)\neq 0$ is  $(d+1)(d+2)/2$, we get from \eqref{eq:defpun} that
 \[
  | v_k(\bfx)|\le \frac {(d+1)(d+2)} 2 (\overline{x}_1-\underline{x}_1) \| \gradi \bfv \|_{L^\infty(\Omega)^{d \times d}},
 \]
 hence concluding the proof.
\end{proof}

We have the following results.

\begin{lemma}\label{lem:estP2conv}
There exists 
$\ctel{cste:estP2conv}$ such that for any  $\bfv = (v_k)_{k=1,\ldots,d}\in C^\infty_c(\Omega)^d$, it holds, for any $K\in\mathcal{T}$,

\begin{equation}\label{eq:estP2inf1}
\| \bfv - \Pi_L \bfv \|_{L^\infty(K)^{d}} \le \cter{cste:estP2conv} h_K \| \gradi \bfv \|_{L^\infty(\Omega)^{d \times d}},
\end{equation}
\begin{equation}\label{eq:estP1inf2}
\| \gradi \Pi_L \bfv \|_{L^\infty(K)^{d \times d}} \le \cter{cste:estP2conv} \| \gradi \bfv \|_{L^\infty(\Omega)^{d \times d}},
\end{equation}
\begin{equation}\label{eq:estP2inf2}
\| \bfv - \Pi_L \bfv \|_{L^\infty(K)^{d}} \le \cter{cste:estP2conv} h_K^2 \| \bfv \|_{W^{2,\infty}(\Omega)},
\end{equation}
and
\begin{equation}\label{eq:estP2gradinf2}
\| \gradi \bfv - \gradi \Pi_L \bfv \|_{L^\infty(K)^{d\times d}} \le \cter{cste:estP2conv} h_K \| \bfv \|_{W^{2,\infty}(\Omega)}.
\end{equation}
\end{lemma}
\begin{proof}
Let $\bfv = (v_k)_{k=1,\ldots,d}\in C^\infty_c(\Omega)^d$.
For $k=1,\ldots,d$, the mean value theorem implies that there exists $\bfx_{i}$ belonging to the segment $[\bfy_i,\bfx]$ such that
\[
 v_k(\bfy_i) - v_k(\bfx) = (\bfy_i - \bfx)\cdot\nabla v_k(\bfx_{i}),
\]
which implies
\[
  |v_k(\bfy_i) - v_k(\bfx)|\le h_K  \| \gradi \bfv \|_{L^\infty(\Omega)^{d \times d}}.
\]

Noticing that, for any $K\in \mathcal{T}$ and for all $\bfx\in K$, we have $\sum_{i\in \mathcal{N}_K\cup\mathcal{E}_K}\phi_i(\bfx) = 1$, we get
\[
 |\Pi_L v_k(\bfx) - v_k(\bfx)|= |\sum_{i\in \mathcal{N}_K\cup\mathcal{E}_K}(v_k(\bfy_i) - v_k(\bfx)) \phi_i(\bfx)|\le \frac {(d+1)(d+1)} 2  h_K  \| \gradi \bfv \|_{L^\infty(\Omega)^{d \times d}}.
\]
This yields \eqref{eq:estP2inf1} as well as

\[
 \nabla v_k(\bfx) =  \sum_{i\in \mathcal{N}_K\cup\mathcal{E}_K}(v_k(\bfy_i) - v_k(\bfx))  \nabla \phi_i(\bfx) =  \sum_{i\in \mathcal{N}_K\cup\mathcal{E}_K}(\bfy_i - \bfx)\cdot\nabla v_k(\bfx_{i})  \nabla \phi_i(\bfx).
\]

Applying \eqref{eq:estgradphiP2}, we conclude that \eqref{eq:estP1inf2} holds. Writing, for $k=1,\ldots,d$ and any $j\in\mathcal{N}_K\cup\mathcal{E}_K$, the Taylor expansion
 \[
   v_k(\bfy_j) =  v_k(\bfx) + (\bfy_j - \bfx)\cdot\nabla  v_k(\bfx) + \frac 1 2 (\bfy_j - \bfx)^t D_2 v_k(\bfx_j) (\bfy_j - \bfx),
 \]
 for some point $\bfx_j$ in the segment $[\bfy_j,\bfx]$, und using \eqref{eq:defpun} in addition to 
 \[
 \sum_{j\in\mathcal{N}_K\cup\mathcal{E}_K} \phi_j(\bfx) = 1\hbox{ and }\sum_{j\in\mathcal{N}_K\cup\mathcal{E}_K} \phi_j(\bfx)(\bfy_j - \bfx) = 0,
 \]
we obtain \eqref{eq:estP2inf2}.
 We deduce, from the preceding Taylor expansion, that
 \begin{multline*}
  \sum_{j\in\mathcal{N}_K\cup\mathcal{E}_K} v_k(\bfy_j)\nabla \phi_i(\bfx) =  \sum_{j\in\mathcal{N}_K\cup\mathcal{E}_K} v_k(\bfx)\nabla \phi_i(\bfx)
  +  \sum_{j\in\mathcal{N}_K\cup\mathcal{E}_K} (\bfy_j - \bfx)\cdot\nabla  v_k(\bfx)\nabla \phi_i(\bfx) \\ + \sum_{j\in\mathcal{N}_K\cup\mathcal{E}_K}\frac 1 2 (\bfy_j - \bfx)^t D_2 v_k(\bfx_j) (\bfy_j - \bfx)\nabla \phi_i(\bfx).
 \end{multline*}
 We notice that, since the interpolation of first order polynomials is exact, 
 \[
  \sum_{j\in\mathcal{N}_K\cup\mathcal{E}_K} v_k(\bfx)\nabla \phi_i(\bfx) = 0\hbox{ and }
   \sum_{j\in\mathcal{N}_K\cup\mathcal{E}_K} (\bfy_j - \bfx)\cdot\nabla  v_k(\bfx)\nabla \phi_i(\bfx) = \nabla  v_k(\bfx),
 \]
 and
 \[
  |\sum_{j\in\mathcal{N}_K\cup\mathcal{E}_K}\frac 1 2 (\bfy_j - \bfx)^t D_2 v_k(\bfx_j) (\bfy_j - \bfx)\nabla \phi_i(\bfx)| \le \ctel{cst:secordappgrad} h_K\| \bfv \|_{W^{2,\infty}(\Omega)},
 \]
which allows to conclude \eqref{eq:estP2gradinf2}.
 \end{proof}

\subsection{Divergence correcting operator}

As stated in the introduction of Section \ref{sec:taylorhood}, we adapt and simplify the edge-bubble operator defined in \cite{Diening2021FortinOF}, without requesting that in the general case the divergence correcting operator vanishes at the boundary of the domain.

\subsubsection{Tangential edge bubble functions}

We define the family  of normalized tangential edge bubble function
$$
\forall \bfx\in\Omega, \ \mathbf{b}_{i,j}(\bfx) = \frac{ (d+2)(d+1)}{  | w_{i,j}|} \varphi_i(\bfx) \varphi_j(\bfx) (\bfy_j - \bfy_i)~\text{for any}~\{i,j\}\in \mathcal{E}.
$$
Note that, using \eqref{eq:supportvarphi} we have
\begin{equation}\label{eq:supportvarphiprod}
\overline{\{ \bfx \in \Omega~\text{such that}~\varphi_i(\bfx) \varphi_j(\bfx) \ne 0 \} }=  \omega_{ij}~\text{for any}~\{i,j\} \in \edges.
\end{equation}

We observe that, for any $\{i,j\} \in \edges$,  we have $\mathbf{b}_{i,j}\in \mathbb{P}^2(\mathcal{T})$, and, in the case where the segment $]\bfy_i,\bfy_j[\subset\Omega$,  we additionally have $\mathbf{b}_{i,j}\in \bfX_{\!N}$ (this property holds even if $\{i,j\} \in \edgesext$).

\begin{lemma}\label{lem:bij}
For any $\{i,j\}\in \mathcal{E}$ and for any $k \in \mathcal{N}$ we have the following relation
$$
( \dive \mathbf{b}_{i,j}, \varphi_k) = \delta_{i,k} - \delta_{j,k}.
$$
\end{lemma}
\begin{proof}
Let $\{i,j\}\in \mathcal{E}$ and let $k \in \mathcal{N}$. Using \eqref{eq:supportvarphiprod} We have 
$$
( \dive \mathbf{b}_{i,j}, \varphi_k) = \int_{\omega_{ij}}  \dive \mathbf{b}_{i,j} \varphi_k \dx.
$$
\begin{itemize}
\item The element $k$ is not an element of $\cup_{K \in \mathcal{T}_{ij}} \mathcal{N}_K$. 
Using \eqref{eq:supportvarphi} we have $\varphi_k(\bfx) = 0$ for any $\bfx \in \omega_{ij}$
and hence the expression vanishes.
\item The element $k$ is an element of $\cup_{K \in \mathcal{T}_{ij}} \mathcal{N}_K$.
We have
$$
( \dive \mathbf{b}_{i,j}, \varphi_k) 
= - \frac{ (d+2)(d+1)}{  | w_{i,j}|} \sum_{ K \in \mathcal{T}_{ij}} \int_K \varphi_i \varphi_j (\bfy_j - \bfy_i) \cdot \nabla\varphi_k \dx $$
$$
+ \frac{ (d+2)(d+1)}{  | w_{i,j}|} \sum_{  K \in \mathcal{T}_{ij}} \sum_{F \in \mathcal{F}_K} \int_F \varphi_k \varphi_i \varphi_j (\bfy_j - \bfy_i)  \cdot \bfn_F \dedge
$$ 
For any $K \in  \mathcal{T}_{ij}$ and $F \in \mathcal{F}_K$ we remark that $ (\bfy_j - \bfy_i) \cdot \bfn_F = 0$ in the case $\{i,j\}\subset \mathcal{N}_F$ and $\varphi_i(\bfx)\varphi_j(\bfx) = 0$ for any $\bfx\in F$ in the case $\{i,j\}\not\subset \mathcal{N}_F$.

We then obtain for any $K \in \mathcal{T}_{ij}$ and for any $F \in \mathcal{F}_K$
$$
 \int_{F} \varphi_k \varphi_i \varphi_j (\bfy_j - \bfy_i)  \cdot \bfn_F \dedge=0
$$
which gives the following identity
$$
( \dive \mathbf{b}_{i,j}, \varphi_k) 
= - \frac{ (d+2)(d+1)}{  | w_{i,j}|} \sum_{ K \in \mathcal{T}_{ij}} \int_K \varphi_i \varphi_j (\bfy_j - \bfy_i) \cdot \nabla\varphi_k \dx.
$$
For $k \notin  \{i, j\}$ using  $\varphi_k (\bfy_i)=0$ and $\varphi_k(\bfy_j)=0$ we obtain that the piecewise constant
function $\nabla\varphi_k$  is orthogonal to the edge vector $(\bfy_j - \bfy_i)$ on each $K \in \mathcal{T}_{ij}$.
Thus, 
the integrals vanishes and the claim follows.
For $k \in \{i,j\}$ using $\varphi_k(\bfy_k)=1$ we obtain for any $K \in \mathcal{T}_{ij}$ the following identity
\begin{equation*} 
\nabla(\varphi_k)_{|K} \cdot (\bfy_j-\bfy_i) = 
\left|\begin{array}{l} \displaystyle
-1, \quad \mbox{for }\ k=i,
\\[2.5ex] \displaystyle
1, \quad \mbox{ for }  \ k=j.
\end{array}\right.
\end{equation*}
For $k=i$ we obtain 
$$
( \dive \mathbf{b}_{i,j}, \varphi_k) =\frac{ (d+2)(d+1)}{  | w_{i,j}|} \sum_{K \in \mathcal{T}_{ij}} \int_K \varphi_i \varphi_j \dx=1,
$$
and for $k=j$ we obtain 
$$
( \dive \mathbf{b}_{i,j}, \varphi_k) = - \frac{ (d+2)(d+1)}{  | w_{i,j}|} \sum_{K \in \mathcal{T}_{ij}} \int_K \varphi_i \varphi_j \dx=-1,
$$
which gives the expected result.
\end{itemize}
\end{proof}

\begin{lemma}\label{lem:boundbij}
 There exists $\ctel{cste:estbijinf}$ such that 
\begin{equation}\label{eq:estbijinf}
\forall K\in\mathcal{T},\ \forall \{i,j\} \in \edges_K,\ \| \mathbf{b}_{ij} \|_{L^\infty(K)^d} \le \cter{cste:estbijinf} h_K^{1-d},
\end{equation}
and
\begin{equation}\label{eq:estgradbijinf}
\forall K\in\mathcal{T},\ \forall \{i,j\} \in \edges_K,\ \| \gradi\mathbf{b}_{ij} \|_{L^\infty(K)^d} \le \cter{cste:estbijinf} h_K^{-d}.
\end{equation}
\end{lemma}
\begin{proof}
We have, for $\{i,j\}\subset \mathcal{N}_K$,
\[
 |\varphi_i(\bfx) \varphi_j(\bfx)| \le 1
\]
and
\[
 | w_{i,j}| \ge \ctel{cte:minmesureomegaij} h_K^d,
\]
which leads to \eqref{eq:estbijinf} since $|\bfy_j - \bfy_i| \le h_K$. From \eqref{eq:estgradivarphi}, we get
\[
 |\nabla(\varphi_i\varphi_j)(\bfx) |\le 2\frac{\cter{cste:estgradvarphi} }{h_K}.
\]
Combining with the preceding inequalities, we deduce \eqref{eq:estgradbijinf}. 
\end{proof}

\medskip

\subsubsection{Divergence correcting operator}

We can now define the divergence correcting operator
 $\Pi_{D} : W_0^{1,\infty}(\Omega)^d \to  \mathbb{P}^2(\mathcal{T})^d $  for any $\bfv \in W_0^{1,\infty}(\Omega)^d$ by
\begin{equation}\label{eq:Piddef}
\forall \bfx\in\Omega,\ \Pi_D v(\bfx) = \sum_{i \in \mathcal{N}} \Pi_{D,i} \bfv (\bfx)\quad \text{with} \quad \Pi_{D,i} \bfv(\bfx) = \sum_{\substack{ j \in \mathcal{N} \\ \{i,j\} \in \edges}}  ( \dive(\varphi_i \bfv), \varphi_j) \mathbf{b}_{j,i}(\bfx).
\end{equation}
Using the fact that $\mathbf{b}_{i,j} = - \mathbf{b}_{j,i}$ for any $\{i,j\} \in \edges$ we remark that this operator satisfies
\begin{equation}\label{eq:Pi2identity}
\forall\bfv \in W_0^{1,\infty}(\Omega)^d,\ \forall \bfx\in\Omega,\ \Pi_D \bfv(\bfx) = \sum_{\{i,j\} \in \edges} (\bfv, - \varphi_i \nabla\varphi_j+ \varphi_j \nabla\varphi_i) \mathbf{b}_{j,i}(\bfx).
\end{equation}
\begin{lemma}\label{lem:pidsupport}
For any $\bfv \in W_0^{1,\infty}(\Omega)^d$ and for any $\bfx\in \Omega$, if $\Pi_D\bfv(\bfx)\ne 0$, then there exists $\bfy\in B(\bfx,\cter{cst:regumesh} h_{\mathcal{T}})$ such that $\bfv(\bfy)\ne 0$ (recall that $\cter{cst:regumesh}$ is introduced in Lemma \ref{lem:reguomegaij}).
\end{lemma}
\begin{proof}
 If $\Pi_D\bfv(\bfx)\ne 0$, there exists $\{i,j\} \in \edges$ such that $(\bfv, - \varphi_i \nabla\varphi_j+ \varphi_j \nabla\varphi_i) \mathbf{b}_{j,i}(\bfx)\neq 0$, which implies $\bfx \in \omega_{ij}$, which yields that there exists $K\in\mathcal{T}_{ij}$ such that $\bfx\in \overline{K}$. This also implies that
 \[
  \int_{\omega_{ij}} |\bfv(\bfy)| {\rm d}\bfy >0.
 \]
 Therefore there exists $\bfy\in  \omega_{ij}$ such that $\bfv(\bfy)\neq 0$.  Remarking that \eqref{eq:diamomegaij} implies that 
   ${\rm diam}(\omega_{ij}) \le \cter{cst:regumesh} h_K$. Hence $|\bfx - \bfy| \le  \cter{cst:regumesh} h_{\mathcal{T}}$.
\end{proof}

We can now prove that this operator preserves the divergence in the following sense.
\begin{lemma}\label{lem:piddiv}
For any $\bfv \in W_0^{1,\infty}(\Omega)^d$ and any $\varphi \in \mathbb{P}^1(\mathcal{T})$ we have
$$
( \dive \Pi_{D} \bfv, \varphi) =( \dive  \bfv,\varphi).
$$
\end{lemma}
\begin{proof}
Using Lemma \ref{lem:bij} we have, for any $k\in\mathcal{N}$,
\begin{equation*}
\forall i\in\mathcal{N},\ ( \dive \Pi_{D,i} \bfv, \varphi_k) 
=\sum_{\substack{ j \in \mathcal{N} \\ \{i,j\} \in \edges}}  ( \dive(\varphi_i \bfv), \varphi_j) ( \delta_{j,k} - \delta_{i,k}).
\end{equation*}
Hence we get $( \dive \Pi_{D} \bfv, \varphi_k) = T_1 + T_2$,
with
\begin{equation*}
T_1 = \sum_{ i \in \mathcal{N}}\sum_{\substack{ j \in \mathcal{N} \\ \{i,j\} \in \edges}}  ( \dive(\varphi_i \bfv), \varphi_j)  \delta_{j,k}\hbox{ and }
T_2 = -\sum_{ i \in \mathcal{N}}\sum_{\substack{ j \in \mathcal{N} \\ \{i,j\} \in \edges}}  ( \dive(\varphi_i \bfv), \varphi_j)  \delta_{i,k}.
\end{equation*}
Remarking that, for any quantity $Q_{i,j}$,
\[
 \sum_{ i \in \mathcal{N}}\sum_{\substack{ j \in \mathcal{N} \\ \{i,j\} \in \edges}}Q_{i,j} = \sum_{ \{i,j\} \in \edges} (Q_{i,j} + Q_{j,i}) = \sum_{ j \in \mathcal{N}}\sum_{\substack{ i \in \mathcal{N} \\ \{i,j\} \in \edges}}Q_{i,j},
\]
we obtain
\[
 T_1 = \sum_{ j \in \mathcal{N}}\sum_{\substack{ i \in \mathcal{N} \\ \{i,j\} \in \edges}}  ( \dive(\varphi_i \bfv), \varphi_j)  \delta_{j,k} =   ( \dive(\sum_{\substack{ i \in \mathcal{N} \\ \{i,k\} \in \edges}}\varphi_i \bfv), \varphi_k).
\]
For any $\bfx\in \omega_k$, we have
\[
 \sum_{\substack{ i \in \mathcal{N} \\ \{i,k\} \in \edges}} \varphi_i(\bfx) = 1 - \varphi_k(\bfx),
\]
which leads to
\[
 T_1 =  ( \dive((1-\varphi_k) \bfv), \varphi_k).
\]

Next we have
\begin{equation*}
T_2 = -\sum_{\substack{ j \in \mathcal{N} \\ \{k,j\} \in \edges}}  ( \dive(\varphi_k \bfv), \varphi_j) =  -( \dive(\varphi_k \bfv), \sum_{\substack{ j \in \mathcal{N} \\ \{k,j\} \in \edges}} \varphi_j).
\end{equation*}
For any $\bfx\in \omega_k$, we have
\[
 \sum_{\substack{ j \in \mathcal{N} \\ \{k,j\} \in \edges}} \varphi_j(\bfx) = 1 - \varphi_k(\bfx),
\]
which implies, since $(\dive(\varphi_k \bfv),1)=0$ (recall that $\bfv$ vanishes on the boundary of $\Omega$),
\begin{equation*}
T_2 =  -( \dive(\varphi_k \bfv),  - \varphi_k) = ( \dive(\varphi_k \bfv), \varphi_k).
\end{equation*}
This yields
\[
 T_1+T_2  = ( \dive(\bfv), \varphi_k),
\]
which concludes the proof.
\end{proof}

\begin{lemma}\label{lem:estP2}
There exists $\ctel{cste:estP2}$ such that, for any $\bfv \in W_0^{1,\infty}(\Omega)^d$, we have
\begin{equation*}
\| \Pi_D \bfv \|_{L^\infty(\Omega)^d} \le \cter{cste:estP2} \| \bfv \|_{L^\infty(\Omega)^d}.
\end{equation*}
\end{lemma}
\begin{proof}
Using \eqref{eq:Pi2identity}  and the definition of the tangential bubble functions we have, for $\bfx\in K$ with $K\in \mathcal{T}$,
\[
| \Pi_D \bfv(\bfx) | \le  \sum_{\{i,j\} \in \edges} | (\bfv, - \varphi_i \nabla\varphi_j+ \varphi_j \nabla\varphi_i)| | \mathbf{b}_{i,j}(\bfx) |.
\]
If $K\subset\omega_{ij}$, then there exists $\ctel{cst:diamomij}$ such that $\omega_{ij}\subset B(\bfx,\cter{cst:diamomij}h_K)$. Since ${\rm m}(\omega_{ij}) \ge\ctel{cst:diamomijb} h_K^d$, we get the existence of $\ctel{cst:diamomijc}$ such that the number of $\{i,j\} \in \edges$ with $\mathbf{b}_{i,j}(\bfx)\neq 0$ is bounded by $\cter{cst:diamomijc}$. Using \eqref{eq:estbijinf} and \eqref{eq:estgradivarphi}, we obtain
\[
 | \Pi_D \bfv(\bfx) | \le \| \bfv \|_{L^\infty(\Omega)^d}\cter{cst:diamomijc}\cter{cste:estbijinf} h_K\frac{\cter{cste:estgradvarphi} }{h_K},
\]
which concludes the proof.
\end{proof}

\begin{lemma}\label{lem:estgradP2}
There exists $\ctel{cste:estgradP2}$ such that
\begin{equation}\label{eq:majgradpidv}
\forall \bfv \in W_0^{1,\infty}(\Omega)^d,\ \| \gradi\Pi_D \bfv \|_{L^\infty(\Omega)^{d\times d}} \le \cter{cste:estgradP2} \max_{K\in\mathcal{T}}\frac {\| \bfv \|_{L^\infty(K)^{d}}} {h_K}.
\end{equation}
\end{lemma}
\begin{proof}
Let us denote, for a given $K\in \mathcal{T}$,
\begin{equation}\label{eq:hypvv}
 C_\bfv = \max_{K\in\mathcal{T}}\frac {\| \bfv \|_{L^\infty(K)^{d}}} {h_K}.
\end{equation}
Using \eqref{eq:Pi2identity}, we have, for $\bfx\in K$,
\begin{equation*}
\gradi \Pi_D v(\bfx) =  \sum_{\{i,j\} \in \edges}  (\bfv, - \varphi_i \nabla\varphi_j+ \varphi_j \nabla\varphi_i) \gradi \mathbf{b}_{i,j}(\bfx).
\end{equation*}
We remark that, using \eqref{eq:hypvv} and \eqref{eq:estgradivarphi} for $K\subset\omega_{ij}$ and $\bfx\in K$,
\[
 |\bfv(\bfx)(-\varphi_i(\bfx) \nabla\varphi_j(\bfx)+ \varphi_j(\bfx) \nabla\varphi_i(\bfx))| \le  \| \bfv \|_{L^\infty(K)^{d}} 2\frac{\cter{cste:estgradvarphi} }{h_K}.
\]
This yields
\[
 |(\bfv, - \varphi_i \nabla\varphi_j+ \varphi_j \nabla\varphi_i)|\le 2|\omega_{ij}| C_\bfv  \cter{cste:estgradvarphi},
\]
and therefore, using \eqref{eq:estgradbijinf}, we conclude \eqref{eq:majgradpidv}.
\end{proof}

\subsection{The operator $\Pi_N $}\label{sec:defPi}

We let  $\bfW = C^\infty_c(\Omega)^d\cap \bfV(\Omega)$. We notice that, for any $\bfv\in\bfW$, then $\Pi_L\bfv\in W^{1,\infty}_0(\Omega)^d$. One can then consider the function $\Pi_D ( \bfv - \Pi_L \bfv)\in \mathbb{P}^2(\mathcal{T})^d$. We notice that this function is not necessarily null on the boundary of the domain, and in this case, it is not an element of $\bfX_{\!N}$. 

We then define $\Pi_N  : \bfW \to \bfX_{\!N}  $, for any $\bfv \in \bfW$, by:
\begin{equation}\label{eq:defpiNTH}
 \hbox{if }\Pi_D ( \bfv - \Pi_L \bfv)\in H^1_0(\Omega)^d,\hbox{ then }\Pi_N  \bfv = \Pi_L \bfv + \Pi_D ( \bfv - \Pi_L \bfv)\hbox{ else }\Pi_N  \bfv = 0.
\end{equation}
Notice that the operator $\Pi_N$ defined by \eqref{eq:defpiNTH} is not linear.
\begin{lemma}\label{lem:divpinzero}
For any $\bfv \in C^\infty_c(\Omega)^d\cap \bfV(\Omega) $ we have $\Pi_N \bfv \in \bfE_N$ that is
\begin{equation}\label{eq:divpinzero}
 \forall \varphi \in \mathbb{P}^1(\mathcal{T}),\ ( \dive \Pi_N  \bfv,\varphi)  =0.
\end{equation}
\end{lemma}
\begin{proof}
 If $\Pi_D ( \bfv - \Pi_L \bfv)\in H^1_0(\Omega)^d$, we get from Lemma \ref{lem:piddiv}
 \[
  ( \dive  \Pi_D ( \bfv - \Pi_L \bfv),\varphi)  =  ( \dive ( \bfv - \Pi_L \bfv),\varphi) =  -( \dive   \Pi_L \bfv,\varphi),
 \]
 which provides \eqref{eq:divpinzero}. Otherwise, we get  \eqref{eq:divpinzero} since $\Pi_N  \bfv = 0$.
\end{proof}

\begin{lemma}\label{lem:estPinf}
There exists $\ctel{cste:estPinf}$ such that for any $\bfv \in \bfW$, we have
\begin{equation}\label{eq:estlinfpin}
\| \Pi_N  \bfv \|_{L^\infty(\Omega)^d} \le \cter{cste:estPinf} \| \gradi \bfv \|_{L^\infty(\Omega)^{d \times d}}.
\end{equation}
and
\begin{equation}\label{eq:estlinfgradpin}
\| \gradi\Pi_N  \bfv \|_{L^\infty(\Omega)^d} \le \cter{cste:estPinf} \| \gradi \bfv \|_{L^\infty(\Omega)^{d \times d}}.
\end{equation}
\end{lemma}
\begin{proof}
It suffices to consider the case  $\Pi_D ( \bfv - \Pi_L \bfv)\in H^1_0(\Omega)^d$ (otherwise the conclusion is straightforward). 
We apply Lemmas \ref{lem:estP1inf} and \ref{lem:estP2}.  We then get \eqref{eq:estlinfpin}.

Now applying Lemmas \eqref{eq:estP2inf1} in Lemma \ref{lem:estP2conv} and Lemma \ref{lem:estgradP2}, we obtain \eqref{eq:estlinfgradpin}.
\end{proof}

\begin{lemma}\label{lem:convpin} There exists $\ctel{cst:convpin}$ such that, for any $\bfv \in \bfW$ such that $\Pi_D ( \bfv - \Pi_L \bfv)\in H^1_0(\Omega)$, then
\begin{equation}\label{eq:estlinfpinpin}
\| \Pi_N  \bfv - \bfv \|_{W^{1,\infty}(\Omega)^d} \le \cter{cst:convpin} h_{\mathcal{T}} \| \bfv \|_{W^{2,\infty}(\Omega)}.
\end{equation}
\end{lemma}
\begin{proof}
We first remark that, from \eqref{eq:estP2inf2} and \eqref{eq:estP2gradinf2} in Lemma \ref{lem:estP2conv}, we get 
 \begin{equation*}
 \| \Pi_L \bfv - \bfv \|_{W^{1,\infty}(\Omega)^d} \le \ctel{cst:convpind} h_{\mathcal{T}} \| \bfv \|_{W^{2,\infty}(\Omega)}.
\end{equation*}
Again applying \eqref{eq:estP2inf2} in the same Lemma, and using Lemma \ref{lem:estgradP2}, we get that
 \begin{equation*}
\| \Pi_D(\bfv - \Pi_L\bfv) \|_{W^{1,\infty}(\Omega)^d} \le \cter{cste:estgradP2}\cter{cste:estP2conv}h_{\mathcal{T}} \| \bfv \|_{W^{2,\infty}(\Omega)},
\end{equation*}
which concludes the proof of \eqref{eq:estlinfpinpin}.
\end{proof}

\begin{theorem}\label{thm:hyppinok}
 Let $\bfW = C^\infty_c(\Omega)\cap\bfV(\Omega)$.
Let $(\mathcal{T}_N)_{N\ge 1}$ be a sequence of simplicial meshes in the sense of this section, such that $(h_{\mathcal{T}_N})_{N\ge 1}$ converges to zero while  $(\theta_{\mathcal{T}_N})_{N\ge 1}$ remains bounded. Let $(\bfX_{\!N} ,M_N,\Pi_N)_{N\ge 1}$ be such that $\bfX_{\!N} $ and $M_N$ are the Taylor-Hood spaces defined by \eqref{eq:defxnmn}, $\Pi_N$ is defined by \eqref{eq:defpiNTH}, in the case where $\mathcal{T} = \mathcal{T}_N$. 
\\
Then the family $\Big(\bfW,(\bfX_{\!N} ,M_N,\Pi_N)_{N\ge 1}\Big)$ satisfies the convergence assumptions \eqref{hyp:PNexist}-\eqref{hyp:Mlim} in Section \ref{sec:cvhyp}.
\end{theorem}
\begin{proof}
The density of $C^\infty_c(\Omega)^d\cap\bfV(\Omega)$ in  $H^1_0(\Omega)^d\cap \bfV(\Omega)$ for the $H^1_0(\Omega)^d$ norm  is proved in particular in \cite[Lemma IV.3.4 p. 249]{BoyerFabrie-book}.

Let $\bfv\in \bfW$. Applying \eqref{eq:defpiNTH}, we have that $\Pi_N\bfv\in \bfE_N$.

Indeed, if $\Pi_D ( \bfv - \Pi_L \bfv)\in H^1_0(\Omega)$, then $\Pi_D ( \bfv - \Pi_L \bfv)\in \bfX_{\!N} $ as well as $\Pi_L \bfv\in \bfX_{\!N} $. In this case, the application of Lemma \ref{lem:divpinzero} yields \eqref{hyp:PNexist}.

Otherwise, if $\Pi_D ( \bfv - \Pi_L \bfv)\notin H^1_0(\Omega)$, then $\Pi_N\bfv = 0\in \bfE_N  $ and  \eqref{hyp:PNexist} is also satisfied.

Let $\bfv\in C^\infty_c(\Omega)\cap\bfV(\Omega)$. Since the support of $\bfv$ is compact, there exists $a>0$ such that, for any $\bfx\in \partial\Omega$, $\bfv = 0$ on the ball $B(\bfx,a)$. Let $N_0 \ge 1$ be such that, for all $N\ge N_0$, $h_{\mathcal{T}_N} \le a/(\cter{cst:regumesh}+2)$.
 Let   $N\ge N_0$, and let us consider $\Pi_L$ defined by \eqref{eq:defpun} in the case  where $\mathcal{T} = \mathcal{T}_N$. Then, for any $\bfx\in \partial\Omega$, $\bfv  - \Pi_L\bfv = 0$ on the ball $B(\bfx,a - h_{\mathcal{T}_N})$.

We then deduce from Lemma \ref{lem:pidsupport} that, for any $\bfx\in\Omega$ with $B(\bfx, h_{\mathcal{T}_N})\cap \partial\Omega\neq\emptyset$, then $\bfv  - \Pi_L\bfv = 0$ on the ball $B(\bfx,\cter{cst:regumesh} h_{\mathcal{T}_N})$, which leads to $\Pi_D ( \bfv - \Pi_L \bfv)\in H^1_0(\Omega)$ for all $N\ge N_0$ in the case  where $\mathcal{T} = \mathcal{T}_N$. Therefore,  $\Pi_D ( \bfv - \Pi_L \bfv)\in H^1_0(\Omega)$ for all $N\ge N_0$ in the case  where $\mathcal{T} = \mathcal{T}_N$. This yields that, for any $N\ge N_0$, the conclusion of Lemma \ref{lem:convpin} holds. This proves \eqref{hyp:PNconv}.

Finally, \eqref{hyp:Mlim} is a consequence of the convergence of $\mathbb{P}^1$ finite elements for the Neumann problem with null average.
 
\end{proof}

\begin{remark}[A pair of finite element which does not fulfill the inf-sup condition]\label{rem:noninfsup}~\\
 Let $\omega>1$ be a given real and let $(\mathcal{T}_N)_{N\ge 1}$  be a sequence of simplicial meshes in the sense of this section, such that $(h_{\mathcal{T}_N})_{N\ge 1}$ converges to zero while  $(\theta_{\mathcal{T}_N})_{N\ge 1}$ remains bounded.  For any  $N\ge 1$, let $\widehat{\mathcal{T}}_N$ be the set of the elements of $\mathcal{T}_N$ whose distance to the boundary is greater than $\omega\ h_{\mathcal{T}_N}$. We then define $\displaystyle\Omega_N = \bigcup_{K\in \widehat{\mathcal{T}}_N} \overline{K}$. We consider $\mathbb{P}^1(\mathcal{T})$ finite element for the pressure, $\mathbb{P}^2(\mathcal{T})$ finite element for the velocity on any $K\in \widehat{\mathcal{T}}_N$, and  $\mathbb{P}^1(\mathcal{T})$  or transition finite element  as in Figure \ref{fig:transelement} for any $K\in \mathcal{T}_N\setminus\widehat{\mathcal{T}}_N$. We then define $\Pi_D$ only for functions whose support is included in $\Omega_N$ and we define $\Pi_N  : \bfW \to \bfE_N = \bfX_{\!N}\cap\bfVN\cap L^\infty(\Omega)^d  $, for any $\bfv \in \bfW$, by:
\begin{equation*}
 \hbox{if {\rm support}}(\bfv - \Pi_L \bfv)\subset \Omega_N,\hbox{ then }\Pi_N  \bfv = \Pi_L \bfv + \Pi_D ( \bfv - \Pi_L \bfv)\hbox{ else }\Pi_N  \bfv = 0.
\end{equation*}
 Following the proof of Theorem \ref{thm:hyppinok}, we can prove that this operator $\Pi_N$ fulfills the convergence assumptions \eqref{hyp:PNexist}-\eqref{hyp:Mlim} in Section \ref{sec:cvhyp}.
 \end{remark}

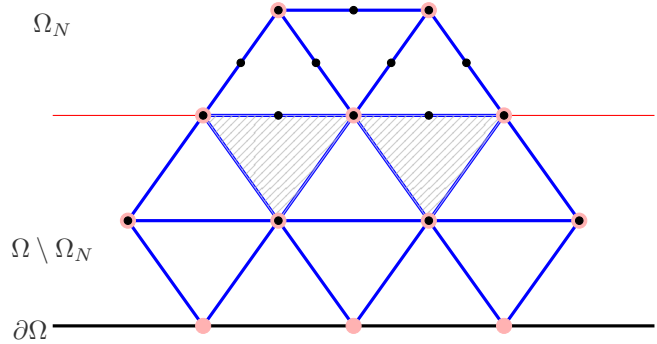
\begin{figure}[!ht]
  \begin{center}
  \begin{tikzpicture}[scale=1]
  
  \path(1,5) node[black!80]{$\Omega_N$};
   
  \path(1,2) node[black!80]{$\Omega\setminus\Omega_N$};
  
  \path(0.7,.95) node[black!80]{$\partial\Omega$};
   
  \draw[-, very thick, black](1,1)--(9,1);
  
  \draw[-, red](1,3.8)--(9,3.8);
   
  \draw[-, very thick, blue](3,1)--(2,2.4);
  \draw[-, very thick, blue](3,1)--(4,2.4);
  \draw[-, very thick, blue](5,1)--(4,2.4);
  \draw[-, very thick, blue](5,1)--(6,2.4);
  \draw[-, very thick, blue](7,1)--(6,2.4);
  \draw[-, very thick, blue](7,1)--(8,2.4);

  \draw[-, very thick, blue](2,2.4)--(4,2.4);
  \draw[-, very thick, blue](4,2.4)--(6,2.4);
  \draw[-, very thick, blue](6,2.4)--(8,2.4);
  \draw[-, very thick, blue](2,2.4)--(3,3.8);
  \draw[-, very thick, blue](4,2.4)--(3,3.8);
  \draw[-, very thick, blue](4,2.4)--(5,3.8);
  \draw[-, very thick, blue](6,2.4)--(5,3.8);
  \draw[-, very thick, blue](6,2.4)--(7,3.8);
  \draw[-, very thick, blue](8,2.4)--(7,3.8);

  \draw[-, very thick, blue](3,3.8)--(5,3.8);
  \draw[-, very thick, blue](5,3.8)--(7,3.8);
  \draw[-, very thick, blue](4,5.2)--(3,3.8);
  \draw[-, very thick, blue](4,5.2)--(5,3.8);
  \draw[-, very thick, blue](6,5.2)--(5,3.8);
  \draw[-, very thick, blue](6,5.2)--(7,3.8);

  \draw[-, very thick, blue](4,5.2)--(6,5.2);

   \draw[-,blue!20, pattern=north east lines, pattern color=black!30, opacity=0.8] (4,2.4)--(5,3.8)--(3,3.8)--(4,2.4);
   \draw[-,blue!20, pattern=north east lines, pattern color=black!30, opacity=0.8] (6,2.4)--(7,3.8)--(5,3.8)--(6,2.4);

   \draw[fill, red!30, opacity=1.](3,1) circle (0.1);
  \draw[fill, red!30, opacity=1.](5,1) circle (0.1);
  \draw[fill, red!30, opacity=1.](7,1) circle (0.1);
  \draw[fill, red!30, opacity=1.](2,2.4) circle (0.1);
  \draw[fill, red!30, opacity=1.](4,2.4) circle (0.1);
  \draw[fill, red!30, opacity=1.](6,2.4) circle (0.1);
  \draw[fill, red!30, opacity=1.](8,2.4) circle (0.1);
  \draw[fill, red!30, opacity=1.](3,3.8) circle (0.1);
  \draw[fill, red!30, opacity=1.](5,3.8) circle (0.1);
  \draw[fill, red!30, opacity=1.](7,3.8) circle (0.1);
  \draw[fill, red!30, opacity=1.](4,5.2) circle (0.1);
  \draw[fill, red!30, opacity=1.](6,5.2) circle (0.1);

  \draw[fill, black, opacity=1.](2,2.4) circle (0.05);
  \draw[fill, black, opacity=1.](4,2.4) circle (0.05);
  \draw[fill, black, opacity=1.](6,2.4) circle (0.05);
  \draw[fill, black, opacity=1.](8,2.4) circle (0.05);
  \draw[fill, black, opacity=1.](3,3.8) circle (0.05);
  \draw[fill, black, opacity=1.](5,3.8) circle (0.05);
  \draw[fill, black, opacity=1.](7,3.8) circle (0.05);
  \draw[fill, black, opacity=1.](4,5.2) circle (0.05);
  \draw[fill, black, opacity=1.](6,5.2) circle (0.05);

%
%
   
   \draw[fill, black, opacity=1.](4,3.8) circle (0.05);
   \draw[fill, black, opacity=1.](3.5,4.5) circle (0.05);
   \draw[fill, black, opacity=1.](6,3.8) circle (0.05);
   \draw[fill, black, opacity=1.](4.5,4.5) circle (0.05);
   \draw[fill, black, opacity=1.](5.5,4.5) circle (0.05);
   \draw[fill, black, opacity=1.](6.5,4.5) circle (0.05);

   \draw[fill, black, opacity=1.](5,5.2) circle (0.05);

  \end{tikzpicture}

  \caption{Finite element pair which does not fulfill the inf-sup condition. The degrees of freedom for the components of the velocity are in black, those for the pressure are in pink. The two hatched elements include 4 degrees of freedom for each of the components of the velocity.}\label{fig:transelement}
  \end{center}
\end{figure}

\section{Has the function $\bar\bfu$ provided by Theorem \ref{thm:cvscheme} some additional regularity?}\label{sec:further}

Theorem \ref{thm:cvscheme} shows the convergence of the incremental projection scheme to a weak solution  $\bar\bfu$ in the sense of Definition \ref{def:weaksol}. The question that $\bar\bfu$ be {\it suitable} in the sense introduced by Caffarelli-Kohn-Nirenberg in \cite{ckn1982par} arises. In \cite{guer2007faedo} and later in \cite{fer2018proof}, the authors proved that the weak solution, obtained as the limit of a semi-discrete or fully discrete numerical scheme coupling the velocity and the pressure under the existence of an inf-sup condition, is a suitable weak solution. Further works have to be done in order to determine whether this property also holds for $\bar\bfu$ obtained as the limit of the projection method.

\bibliographystyle{abbrvurl}

\end{document}